\documentclass[12pt]{amsart}

\usepackage{fullpage}
\usepackage{verbatim}
\usepackage{amsmath,amssymb}
\usepackage{color}
\usepackage[all]{xy}
\usepackage{tikz-cd}
\usepackage{enumitem}
\setlist[enumerate]{label={\upshape(\roman*)}}

\usepackage{hyperref}

\newtheorem{theorem:intro}{Theorem}
\newtheorem{proposition}{Proposition}[section]
\newtheorem{corollary}[proposition]{Corollary}
\newtheorem{lemma}[proposition]{Lemma}
\newtheorem{theorem}[proposition]{Theorem}

\theoremstyle{definition}
\newtheorem{definition}[proposition]{Definition}
\newtheorem{example}[proposition]{Example}
\newtheorem*{remark}{Remark}
\newtheorem*{notation}{Notation}

\newcommand{\ZZ}{{\mathbb Z}}

\newcommand{\NN}{{\mathbb N}}

\newcommand{\im}{\mathrm{im}}

\newcommand{\Aut}{{\mathsf{Aut}}}
\newcommand{\Sym}{{\mathsf{Sym}}}

\newcommand{\Triv}{{\mathsf{Triv}}}
\newcommand{\St}{{\mathsf{St}}}
\newcommand{\rst}{{\mathsf{Rst}}}
\newcommand{\aug}{\mathrm{aug}}
\newcommand{\Aug}{\mathrm{Aug}}

\newcommand{\qqand}{{\qquad\text{and}\qquad}}

\newcommand{\cl}[1]{{\overline{#1}}}

\begin{document}

\title{Branch groups with infinite rigid kernel}

%\author[1]{Alejandra Garrido}
%\author[2,3]{Zoran \v{S}uni\'c}
%\affil[1]{Facultad de Matem\'aticas, Universidad Complutense de Madrid and ICMAT, Madrid, SPAIN\\ alejandra.garrido@ucm.es, alejandra.garrido@icmat.es}
%\affil[2]{Department of Mathematics, Hofstra University, Hempstead NY 11549, USA \\ zoran.sunic@hofstra.edu}
%\affil[3]{UKIM, Skopje, North Macedonia}
%\renewcommand\Authands{ and }

\author{Alejandra Garrido}
\address{Facultad de Matem\'aticas, Universidad Complutense de Madrid, and ICMAT, Madrid, SPAIN}
\email{alejandra.garrido@ucm.es; alejandra.garrido@icmat.es}

\thanks{The first author received financial support from Spain’s Ministry of Science and Innovation [PID2020-114032GB-I00]; and the Severo Ochoa Programme for Centres of Excellence in Research and Development [CEX2019-000904-S]}

\author{Zoran \v{S}uni\'c}
\address{Department of Mathematics, Hofstra University, Hempstead NY 11549, USA, and 
UKIM, Skopje, North Macedonia}
\email{zoran.sunic@hofstra.edu}

\thanks{Both authors are grateful to the anonymous referee for carefully reading a previous version of this article, suggesting many improvements and correcting mistakes.}

\dedicatory{To Slava Grigorchuk, in gratitude and acknowledgement of the many ramifications of his work.}

\begin{abstract}
A theoretical framework is established for explicitly calculating rigid kernels of self-similar regular branch groups. 
This is applied to a new infinite family of branch groups in order to provide the first examples of self-similar, branch groups with infinite rigid kernel.
The groups are analogues of the Hanoi Towers group on 3 pegs, based on the standard actions of finite dihedral groups on regular polygons with odd numbers of vertices, and the rigid kernel is an infinite Cartesian power of the cyclic group of order 2, except for the original Hanoi group. 
The proofs rely on a symbolic-dynamical approach, related to finitely constrained groups.
\end{abstract}


\maketitle

%\todo{\begin{itemize}
%	\item congruence completion is finitely constrained, defined by patterns of size 2 -- can prove this using calculations for rigid kernel. put afterwards (or find alternative proof)
%	\item infinite quotient of $D$
%	\item can anything be said about the rigid kernel in terms of $|X\ast G:G_1|$ or similar? [prob not...]
%	\item comments on surjective inverse system for rigid kernel: all previous examples fail this, but the rigid kernel is trivial anyway because $K\geq G_n$ for some $n$
%	\item intro
%	\item abstract
	
%	
%	\end{itemize}
%}

\section{Introduction}
A group $G$ with a faithful action on a level-homogeneous rooted tree $T$ is a \emph{branch group} if it acts level-transitively and for all $n$, the rigid stabiliser $\rst_G(n)=\prod_{v\in \text{level } n}\rst_G(v)$ has finite index in $G$, where $\rst_G(v)$ consists of all elements of $G$ that are only supported on the subtree rooted at $v$.

Since the appearance in the 1980s of the first finitely generated examples, branch groups have been recognized as an important class of groups. 
This is not only because of the many examples in the class with interesting properties:  finitely generated infinite torsion groups, groups of intermediate word growth,  of non-uniformly exponential word growth, amenable but not elementary amenable, etc..(\cite{grig_burnside}, \cite{grig_intermediate}, \cite{Wilson_nonuniform}), but also because the subgroup structure of branch groups forces them to appear as cases in classifications. 
For instance, by a result of Wilson (\cite{wilsonJIclassification}, see also \cite{GrigNewHorizons}), all residually finite just infinite groups (infinite groups all of whose proper quotients are finite) are either just infinite branch groups, virtually direct powers of non-abelian hereditarily just infinite  groups, or virtually abelian. 
This classification also applies in the case of just infinite profinite groups (all quotients by closed normal subgroups are finite), \cite{wilsonNewhorizons}. 
Profinite branch groups also arise as compact-open subgroups of \emph{locally decomposable} groups, one of five types into which one can separate the groups that are totally disconnected, locally compact, compactly generated, topologically simple and non-discrete (\cite{CRW-Part2}).

If $G$ is a branch group, it must be residually finite, so it embeds densely into its profinite completion $\widehat{G}$. 
The action of $G$ on the rooted tree $T$ gives another completion, with respect to the topology generated by the stabilisers $\St_G(n)$ of levels of the tree; this completion, denoted $\cl{G}$, is in fact the closure of $G$ in the permutation topology of $\Aut(T)$, the (profinite) group of all automorphisms of $T$. 
The rigid level stabilisers, $\rst_G(n), n\in \NN$, being of finite index, also yield a profinite topology on $G$, and $G$ also embeds densely into the completion $\widetilde{G}$ with respect to this topology. 
Since $\rst_G(n)\leq \St_G(n)$ for every $n\in \NN$, and they are all of finite index in $G$, there are surjective morphisms
$$\widehat{G}\rightarrow \widetilde{G} \rightarrow \cl{G}$$
and the \emph{congruence subgroup problem} asks what are the kernels:
the \emph{congruence kernel} $\ker(\widehat{G}\rightarrow \cl{G})$, the \emph{branch kernel} $\ker(\widehat{G}\rightarrow \widetilde{G})$ and the 
\emph{rigid kernel} $\ker(\widetilde{G}\rightarrow \cl{G})$. 

The term ``congruence'' is used by analogy with the classical congruence subgroup problem for $\mathrm{SL}_n(\mathbb{Z})$,  from which these questions take inspiration.
The problem of determining the congruence, branch and rigid kernels for a branch group $G$ was first posed in \cite{bartholdi-s-z:congruence}, where the first systematic study of this problem was undertaken. 
It was determined in \cite{garridoCSP} that the congruence subgroup problem is independent of the branch action of $G$. %(the congruence subgroup problem for arithmetic groups is also independent of the linear embedding). 

There are now various examples of branch groups in the literature that show different behaviours for the various kernels. 
For example, Grigorchuk groups and GGS-groups all have trivial congruence kernel (we say that they have the \emph{congruence subgroup property}), \cite{GrigNewHorizons, agu_ggs}. 
Pervova constructed in  \cite{pervova:completions} the first examples of branch groups with non-trivial (in fact, infinite) congruence kernel. 
The rigid kernel of these groups is trivial. 
The first example of a branch group shown to have a non-trivial rigid kernel is $H^{(3)}$, the Hanoi Towers group on 3 pegs \cite{grigorchuk-s:hanoi-cr}. 
It was determined in \cite{bartholdi-s-z:congruence} and, later, using a different approach, in \cite{skipper:congruence}, that the rigid kernel is the Klein 4-group. 
An infinite family of branch groups with non-trivial rigid kernel is constructed by Skipper in \cite{skipper:congruence}, but the rigid kernel for these examples is not determined and the groups come without self-similar actions (they are constructed as subgroups of some self-similar groups). 

In this paper, we provide the first known examples of self-similar, branch groups with infinite rigid kernel. 
They are an infinite family of ``dihedral generalizations" of the Hanoi Towers group $H^{(3)}$, one for each odd $d\geq 3$, acting on the $d$-regular rooted tree. 
The rigid kernel is determined explicitly, it is isomorphic to a Cartesian product of cyclic groups of order 2. 
The product is infinite for $d\geq 5$ and, of course, of rank 2 if $d=3$.

\subsection*{Main examples: Hanoihedral groups}
Let $d =2k+1$ be an odd integer, with $d \geq 3$ ($k \geq 1$), and $X=\{0,1,\dots,d-1\}$. For $i = 0,\dots,d-1$, let $\mu_i$ be the involution in $\Sym(X)$ given by 
	\[
	\mu_i = (i-1~i+1)(i-2~i+2) \dots (i-k~i+k),
	\]
where the entries in the inversion pairs are considered modulo $d$ (for instance, for $d=5$, $\mu_4 = (35)(26) = (03)(12)$). Note that $\mu_i$ can be interpreted as the mirror symmetry of the regular $d$-gon with vertices $0,1,\dots,d-1$ with respect to the axis through vertex $i$. The group generated by all $\mu_i$, $i=0,\dots,d-1$, is the dihedral group $D(d)$, the group of symmetries of the regular $d$-gon. For $i,j=0,\dots,d-1$, we have $\mu_i\mu_j = \rho^{2(i-j)}$, where $\rho=(012\dots d-1)$ is the rotation of the $d$-gon by $2\pi/d$. We also have $\mu_i\rho\mu_i=\rho^{-1}$ and $\mu_i\rho = \rho \mu_{i-1}$ (indices modulo $d$). The dihedral group $D(d)$ consists of $d$ mirror symmetries, $\mu_0,\dots,\mu_{d-1}$, and  $d$ rotations, $1,\rho,\rho^2, \dots,\rho^{d-1}$. The rotations form a subgroup of index 2 in $D(d)$, which is also the commutator subgroup of $D(d)$.

For $i=0,\dots,d-1$, let $a_i$ be the automorphism of the $d$-ary rooted tree $X^*$ given by 
	\[
	a_i = \mu_i(1,\dots,1,a_i,1,\dots,1), 
	\]
where the only non-trivial section appears in coordinate $i$ (the only coordinate not moved by $\mu_i$). Let $D$ be the self-similar group
	\[
	D = \langle a_0,\dots,a_{d-1} \rangle. 
	\]
Note that we defined one group for every odd integer $d \geq 3$, but this is not reflected in our notation for $D$ or its generators (we do not index them by $d$) lest the notation become cumbersome. 
	
%	When $d=3$ the group $D$ is the Hanoi Towers group $H^{(3)}$ on 3 pegs. The other groups in the family have not been explicitly considered in the literature. 
The definition of the generators of $D$ fits the general construction of generators for ``Hanoi-like groups'' from a set of permutations, mentioned in~\cite[Example 4]{grigorchuk-s:standrews}.
Namely, the Hanoi Towers group $H^{(d)}$, $ d \geq 3$, is generated by the automorphisms corresponding to all transpositions $(ij)$, the group $D$ is generated by the automorphisms corresponding to all mirror permutations $\mu_i$, while the examples in~\cite{skipper:congruence} by the automorphisms corresponding to all $d$ cycles of length $d-1$ obtained from the cycle $(0 1 \dots d-1)$ by removing one letter. More generally, all of these examples may be seen as variations inspired by the groups constructed by P.~Neumann~\cite[Section 5]{neumann:pride}. However, the end result is very different, as P.~Neumann's examples are just-infinite groups with very simple branching structure and trivial congruence kernel, while the group $D$ has a non-trivial rigid  kernel.  
% and is not just-infinite. 

The following hold for these examples.
Here, $X\ast D'$ denotes $\prod_{v\in X}\delta_x(D')\leq \Aut(X^*)$ where $\delta_x(g)$ is the element of $\Aut(X^*)$ that acts as $g\in \Aut(X^*)$ on the subtree rooted at $x\in X$ and fixes the rest of $X^*$. 
 See Section \ref{sec:2} for more definitions.

\begin{theorem:intro}\label{t:list}
Let $d$ be an odd integer, with $d \geq 3$. Then:
	
\begin{enumerate}
	\item $D$ is a self-similar, self-replicating, contracting, regular branch group, branching over its commutator subgroup $D'$. 
		
%	\item A word $w$ over $\{a_0,\dots,a_{d-1}\}$ represents an element in the commutator $D'$ if and only if each of the letters $a_i$ appears even number of times in $w$.  
		
\item $D/\St_D(1) = D(d)$. 
		
%	\item A word $w=a_{i_m}a_{i_{m-1}} \dots a_1$ over $\{a_0,\dots,a_{d-1}\}$ represents an element in the first level stabiliser $\St_D(1)$ if and only if it has even length and $i_1+i_3+ \dots i_{m-1} \equiv_d i_2+i_4 + \dots +i_m$.  
		
\item The group $D/D'$ is the elementary abelian 2-group of rank $d$ and 
    \[
	|D':X\ast D'| = d \cdot 2^{(d-1)(d-2)}. 
    \]

	\item The rigid level stabiliser $R_n$ of level $n$ is $X^{n}\ast D'$.
    For $n \geq 1$, we have 
	\[
	|D:R_n| = 2^{(d-2)d^n+2} \cdot d^{\frac{d^{n}-1}{d-1}}. 
	\]
		
%		\item The group $D$ is not just-infinite. %TODO: prove not j.i.
		
	\item 	The rigid kernel of $D$ is an elementary abelian 2-group
    \[
    A\times \prod_{X^*\setminus\epsilon}B, 
    \]
	where $A$ has rank $(d-1)(d-2)$ %\cong (\ZZ/2\ZZ)^{(d-1)(d-2)}$ 
	and $B$ has rank $(d-1)(d-3)$. %	\cong (\ZZ/2\ZZ)^{(d-1)(d-2)}$. 
	In particular, the rigid kernel is a Klein-4 group when $d=3$ and isomorphic to $\prod_{\NN}\ZZ/2\ZZ$ when $d>3$.

	\item The closure $\cl{D}$ is a finitely constrained group defined by the  patterns of size $2$ that can be described as follows. A pattern of size 2 
	\[
	\xymatrix{
		&& \pi 
		\ar@{->}_{0}[lld] \ar@{->}^{1}[ld] \ar@{->}_{d-2}[rd] \ar@{->}^{d-1}[rrd]&& \\
		\pi_0 & \pi_1 & \dots & \pi_{d-2} & \pi_{d-1} 
	}
	\]
	is an allowed pattern if and only if the permutation $\pi\pi_0\pi_1\dots \pi_{d-2}\pi_{d-1}$ is a rotation in $D(d)$ (that is, the number of mirror symmetries among $\pi,\pi_0,\dots,\pi_{d-1}$ is even).

	\item The closure $\cl{D}$ is a regular branch group branching over $\cl{D}_1$, the stabiliser of level 1. We have, for $n \geq 1$,  
	\[
	|D:\St_D(n)| = |\cl{D}:\cl{\St_D}(n)| = 2^{d^{n-1}} \cdot d^{\frac{d^{n}-1}{d-1}}.
	\]
	and the Hausdorff dimension of $\cl{D}$ is 
	\[
	1- \frac{1}{d} \cdot \frac{\log 2}{\log 2d}. 
	\]
		
%	\item \label{i:nontrivial}  We have 
%	\[
%	[X\ast D :\St_D(1)] = 2^{2d-1} \qqand [X\ast\cl{D} :\cl{D}_1] =2. 
%	\]
%	In particular, by Theorem~\ref{t:criterion}, $D$ has a nontrivial rigid kernel. 

\end{enumerate}
\end{theorem:intro}

Notice that since the rigid kernel is elementary abelian, the branch completion $\widetilde{D}$ of $D$ is not a branch group, because these do not contain non-trivial virtually abelian normal subgroups. 
In particular, the rigid kernel will be in the kernel of any branch action of $\widetilde{D}$. 

%\begin{theorem}\label{t:rigid_kernel_D}
%
%\end{theorem}

In order to make our calculations, we also develop a theoretical framework for finding the structure of rigid kernels.
As for most of the studied examples of branch groups, we concentrate on the case of self-similar groups (see Section \ref{sec:2} for definitions) because in that case we can exploit the symbolic-dynamical results obtained in \cite{penland-sunic:kitchens} and \cite{sunic:pibonacci} for \emph{finitely constrained groups}. 
In particular, we obtain the following criterion for determining whether or not the rigid kernel is trivial. 

\begin{theorem:intro}\label{t:criterion}
	Let $G \leq \Aut(X^*)$ be a self-similar, level-transitive, regular branch group, with maximal branching subgroup $K$. 
	The following are equivalent:
	\begin{enumerate}
		\item $G$ has  trivial rigid kernel;
		\item $G$ branches over some level stabiliser;
		\item $K\geq \St_G(n)$ for some $n\geq 1$.
	\end{enumerate}
	
	%		\item $G$ branches over its second level stabiliser;
	If $G$ is in addition self-replicating, then $\cl{G}$ is a  finitely constrained group and the above items are also equivalent to 
	
	\begin{enumerate}
		\setcounter{enumi}{3}
		\item $ |X\ast G:\St_G(1)| =  |X\ast \cl{G}:\cl{\St_G(1)}|$.
	\end{enumerate}
\end{theorem:intro}

 \begin{example}
	It is known that, for $G = H^{(3)}$, the Hanoi Towers group on 3 pegs,   we have
	\[
	|G\times G \times G:G_1| = 2^5 \qqand |\cl{G}\times \cl{G} \times \cl{G}:\cl{\St_G}(1)|=2. 
	\]
	Thus, $H^{(3)}$ has a non-trivial rigid kernel. 
	
	The value of the former index was indicated in~\cite{grigorchuk-s:standrews}, while the value of the latter can be easily inferred from the description of the closure $\cl{H^{(3)}}$ as a finitely constrained group, announced in~\cite{grigorchuk-n-s:oberwolfach2}. These indices were mentioned and used in~\cite{sunic:pibonacci} in the context of the calculation of the Hausdorff dimension of $\cl{H^{(3)}}$ (which happened to be $1-1/3\log_6 2$), along with a remark that relates the fact that these indices are different to the fact that $H^{(3)}$ is not branching over any level stabiliser, which is to say, in the terminology of~\cite{bartholdi-s-z:congruence}, that $H^{(3)}$ has non-trivial rigid kernel. 
	%The present text may be thought of an attempt to elucidate that remark, formalize it into a general result (Theorem~\ref{t:criterion}), and show how it can be used in other situations, by providing infinitely many examples of groups with nontrivial rigid kernel (Theorem~\ref{t:list}\ref{i:nontrivial}). 
	The present text grew out of an attempt to elucidate and formalize that remark. 
\end{example}

Under the assumption that the group is self-similar and regular branch, we obtain various results on the size of the rigid kernel, the most explicit and simple to state of which is as follows.

\begin{theorem:intro}\label{thm:rigid_kernel_surjective_system_direct_prod}
	Let $G$ be a self-similar, level-transitive regular branch group with maximal branching subgroup $K$. 
	Suppose that $\St_G(2)K\geq \St_G(1)$. 
	\begin{enumerate}
		\item The rigid kernel of $G$ is trivial if and only if $\St_G(2)=\Triv_G(2)$. 
		%		\item It is infinite if and only if $\St_G(2)\cap K \gneq \Triv_G(2)$. 
		\item Suppose moreover that $G/K$ is in a class of groups that is closed under subgroups, quotients and direct products, in which all short exact sequences that are in the class split as direct products (for example, elementary abelian groups). 
		Then the rigid kernel is  $$\Gamma=\frac{\St_G(2)}{\Triv_G(2)}\times \prod_{X^*\setminus\{\epsilon\}} \frac{\St_G(2)\cap K}{\Triv_G(2)}$$
		and it is infinite if and only if $\dfrac{\St_G(2)\cap K}{\Triv_G(2)}$ is non-trivial.
	\end{enumerate}
\end{theorem:intro}

Section \ref{sec:2} contains definitions and necessary prerequisites and the proofs of the above theorems. 
Section \ref{sec:3} is devoted to the Hanoihedral groups and proving the items in Theorem \ref{t:list}.

%--------------------------------------------------------------
\section{Symbolic portraits and branch completions}\label{sec:2}

\subsection{Tree-shifts}
Let $A$ be a set.
Symbolic dynamics has for a long time concerned the study of one-sided sub-shifts: closed subsets of $A^{\mathbb{N}}$ with the product topology (where $A$ is discrete) that are invariant under the shift map $\cdot_{1}:A^{\mathbb{N}}\rightarrow A^{\mathbb{N}}, f_{1}(n)=f(1+n)$. 
This shift map extends to a right action of $\mathbb{N}$ on $A^{\mathbb{N}}$.

Now, $\NN$ is the free monoid on a single generator: the set of all finite strings over the alphabet $\{1\}$ with operation concatenation.
This of course generalises to the free monoid over an alphabet $X$ (which we always take to be finite): the set $X^*$ of all finite strings over $X$.
The right Cayley graph of $X^*$ is an infinite,  regular, rooted tree whose vertices of level $n$ are the words $X^n$ over $X$ of length $n$.

A way of visualising an element of $A^{X^*}$ is as a \emph{tree portrait}: every vertex of the right Cayley graph of $X^*$ is decorated by a value in $A$.

Passing from  $\NN$ to $X^*$ we obtain:

\begin{definition}\label{def:tree-shift-self-sim-repl-over-A}
A subset $S\subseteq A^{X^*}$ is \emph{self-similar over $A$} if it is invariant under the shifts $\cdot_x:A^{X^*}\rightarrow A^{X^*}$, for $x\in X$, defined by $f\mapsto f_x$  where $f_{x}(u)=f(xu)$.
If the shifts are onto $S\subseteq A^{X^*}$, then $S$ is \emph{self-replicating over $A$}.

Extending these maps to all of $X^*$ yields a right action of $X^*$ on  $A^{X^*}$.

A \emph{tree-shift over the alphabet $A$} is a closed subset $S\subseteq A^{X^*} $ that is self-similar over $A$.
\end{definition}

Tree-shifts have been studied in  \cite{cech-c-f-s:cellautomata} and \cite{penland-sunic:kitchens}.

Suppose now that $A$ is a finite group.
Then $A^{X^*}$ is a compact totally disconnected (i.e., profinite) group, whose basic open neighbourhoods of the identity are the cylinder sets
$$\Triv^A(n):=\{f\in A^{X^*}: f(u)=e_A \text{ for all } u\in X^{< n}\}, \text{ where } X^{<n}:=\bigcup_{i=0}^{n-1}X^i, \, n\geq 1.$$
The subspace topology of $G\leq A^{X^*}$ has $\Triv^A_G(n):=G\cap \Triv^A(n), \, n\geq 1$ as basic identity neighbourhoods.

\begin{notation}
	 For an element $g\in A^{X^*}$  and $v\in X^{*}$ we write $\delta_v(g)$ to mean the element $f\in A^{X^*}$ such that: 
	 \begin{enumerate}
	 	\item $f_v=g$,
	 	\item $f(u)$ is trivial for all prefixes $u$ of $v$,
	 	\item $f_w$ is trivial for all $w\in X^*$ that is not comparable with $v$.
	 \end{enumerate}
Given a subgroup $H\leq A^{X^*}$ and $n\in\NN$, define $X^n\ast H:=\prod_{v\in X^n}\delta_v(H)\leq \Triv^A(n)$.
\end{notation}

\begin{definition}\label{def:branchsymbolically}
	A subgroup $G\leq A^{X^*}$ that is self-similar over $A$ is said to \emph{branch symbolically over $H\leq G$} if $H\geq X\ast H$. 
In particular, $$\Triv^A_G(1)\geq\Triv^A_H(1)\geq X\ast H\geq X\ast \Triv_H^A(1)=\Triv^A_H(2).$$ 
\end{definition}

The most obvious example of a group $A$ to take is $A=\Sym(X)$.
Then the full shift $\Sym(X)^{X^*}$ can be viewed as the group $\Aut(X^*)$ of rooted tree automorphisms
    \[ \Aut(X^*)%=\varprojlim_{n\geq 1}\dfrac{\Aut(X^*)}{\St(n)}
    = \varprojlim_{n\geq 1} \Aut(X^{\leq n}) = \varprojlim_{n\geq 1} \overbrace{\Sym(X)\wr \dots \wr\Sym(X)}^n,
    \]
where $\Aut(X^{\leq n})=\overbrace{\Sym(X)\wr \dots \wr\Sym(X)}^n$ is the automorphism group of the finite subtree of $X^*$ consisting of words of length at most $n$. Multiplication and inversion of portraits are the operations in an infinitely iterated wreath product: if $f,g \in \Sym(X)^{X^*}$ then 
\begin{align*}
    fg(\epsilon)&=f(\epsilon)g(\epsilon)\\
    fg(x)&=f(g(\epsilon)(x))g(x),~x\in X\\
    %fg(x_1x_2)&=f(g(\epsilon)(x_1)\,g(x_1)(x_2))g(x_1x_2), \;x_1,x_2\in X\\
    &\dots\\
    fg(x_1x_2\dots x_n)&=f(y_1y_2\dots y_n)g(x_1x_2\dots x_n)
    \text{ where } y_i=g(x_1 x_2 \dots x_{i-1})(x_i),~x_i\in X.
\end{align*}
The element $y_1\dots y_n\in X^*$ is the image of $x_1\dots x_n\in X^*$ under the action of $g$. 
Writing $g[u]$ for the image of $u\in X^*$ under $g\in \Aut(X^*)$, the last line above becomes
 \[
 fg(u)=f(g[u])g(u)\in \Sym(X).
 \]
 In this setting, $\Triv^A(n)=\St(n)=\bigcap_{v\in X^n}\St(v)$, the stabiliser of all vertices of length at most $n$, 
 and $\Aut(X^*)/\St(n)\cong \Aut(X^{\leq n})$.
The definition of multiplication ensures that the shift maps $\cdot_x:\Aut(X^*)\rightarrow \Aut(X^*)$ are partial endomorphisms (endomorphisms when restricted to $\St(x)$).
%
%Note that the product topology on $\Sym(X)^{X^*}$ coincides with the profinite topology on $\Aut(X^*)$  whose basic identity neighbourhoods are the level stabilisers 
%    \[
%    \St(n)=\{f\in\Aut(X^*): f(x_1\dots x_m)=\mathrm{id} \text{ for all } m < n,~x_i\in X\}.
%    \]
%This yields $\Aut(X^*)=\varprojlim_{n\geq 1}\dfrac{\Aut(X^*)}{\St(n)}=\varprojlim_{n\geq 1}\Aut(X^n)$.

%To avoid confusion, for an automorphism $f\in \Aut(X^*)$ we will write $f(u)$ for the image of $u\in X^*$ under the action of $f$ and $f_{(u)}\in A$ for the decoration of $f$ at $u$.
%For each $x\in X$, the image of the decoration $f_{(u)}\in A$ under the shift map $\cdot_x$ is then simply $f_{(ux)}\in A$.

When $A=\Sym(X)$ we do not mention $A$ in the definition of self-similar or branching symbolically:
\begin{definition}\label{def:self-sim_aut}
A \emph{self-similar group} is a subgroup $G$ of $\Aut(X^*)$  that is self-similar over $\Sym(X)$, i.e., invariant under the right action of $X^*$ on $\Aut(X^*)$ given by the shift maps $\cdot_w: \Aut(X^*)\rightarrow \Aut(X^*), f_w(u)=f(wu)$, for $w\in X^*$. 
The group $G$ is  \emph{self-replicating} if it is self-replicating over $\Sym(X)$.

A \emph{closed self-similar group} is a closed subgroup of $\Aut(X^*)$ that is also self-similar (i.e. a subgroup of $\Aut(X^*)$ that is also a tree-sub-shift). 

A self-similar group $G\leq \Aut(X^*)$ \emph{branches over $K\leq G$} if $X\ast K \leq K$. 
It is a \emph{regular branch} group if it branches over some finite index normal subgroup. 
\end{definition}

Given a self-similar group  $G\leq \Aut(X^*)$, its closure $\overline{G}\leq \Aut(X^*)$ is a closed self-similar group, called its \emph{congruence completion}.
It can also be seen as the inverse limit 
    \[
    \overline{G}=\varprojlim_{n\in\NN}G/\St_G(n)
    \]
of the inverse system  of groups $\{G/\St_G(n), n\in \NN \}$ and canonical morphisms $\{s_{m,n}:G/\St_G(m)\rightarrow G/\St_G(n), m\geq n \in \NN\}$ where $\St_G(n)=\St(n)\cap G$.

%\begin{notation}
%Given an element $g\in \Aut(X^*)$  and $v\in X^{*}$ we write $\delta_v(g)$ to mean the element $f\in \Aut(X^*)$ such that $f_v=g$ and $f_u$ is trivial for all $u\notin vX^*$. 
%%  $X^n\ast g=(g_u)_{u\in X^n}$ and think of it as an element of $\Aut(X^*)$ that has trivial decorations for all $v\in X^{< n}$.
%For a subgroup $K\leq \Aut(X^*)$, and $n\in \NN$, denote 
%$X^n\ast K:=\{ \prod_{v\in X^n} \delta_v(k_v) : k_v\in K  \} \leq \St(n)$.
%\end{notation}
%
%\begin{definition}\label{def:reg-branch_aut}
%A self-similar group $G$ is \emph{regular branch} if it contains some normal, finite-index subgroup $K$ such that $X\ast K \leq K\leq G$.
%In that case, it is also said that $G$ is \emph{regular branch over $K$} or that \emph{$G$ branches over $K$}.
%\end{definition}

Many of the most studied examples of self-similar groups are regular branch. 
For such groups, it makes sense to consider the larger tree-shift space  $(\Sym(X)\times G/K)^{X^*}$ where now $A=\Sym(X)\times G/K\geq G/(\St(1)\cap K)$. 
Note that we can keep the previous wreath product structure, by letting all factors $G/K$ act trivially, so that $(\Sym(X)\times G/K)^{X^*}$ becomes the group $S_K:= \Aut(X^*) \ltimes (G/K)^{X^*}$. 
Recalling that we denoted by $g[u]$ the image of vertex $u\in X^*$ under $g\in\Aut(X^*)$, multiplication in $S_K$ is defined by 
    \[
    (f,h)(g,l)=(fg,h^gl)
    \]
where $h^g(u)=h(g[u])\in G/K$ and multiplication in $G/K^{X^*}$ is performed component-wise. 
Again, this definition ensures that all shift maps $\cdot_x: S_K\rightarrow S_K$ for $x\in X$, are partial endomorphisms.

For this case, we will denote $\Triv^{\Sym(X)\times G/K}(n)$ simply by $\Triv^K(n)$. 
Having identified $(\Sym(X)\times G/K)^{X^*}$ with $S_K=\Aut(X^*)\ltimes (G/K)^{X^*}$, these subgroups become
$$\Triv^K(n)=\{f\in S_K: f(u)=(\mathrm{id},K) \text{ for all } u\in X^{< n}\}.$$

%The product topology on $(\Sym(X)\times G/K)^{X^*}$ also turns $S_K$ into a profinite group, with a base of neighbourhoods of the identity given by 
%    \[ 
%    \Triv^K(n)=\{f\in S_K: f(u)=(\mathrm{id},K) \text{ for all } u\in X^{< n}\}, \; n\geq 1.
%    \]
%    
    
%This fits the chosen notation for $\St(n)$, so that $\Triv(n)\leq \St(n)$ for all $n$ and, indeed, we have 
% $$\Triv^K(n)=\St(n)\cap X^{n-1}\ast K=X^{n-1}\ast (\St(1)\cap K).$$

%The following definition is entirely analogous to Definitions~\ref{def:self-sim_aut} and~\ref{def:reg-branch_aut}. 
%Although the notation is the same as for subgroups of $\Aut(X^*)$, it should be clear from context which notion is meant. 
%
%
%\begin{definition}\label{def:self-sim_symbolic}
%Analogously to $\Aut(X^*)$, we will say that a subgroup of $S_K$ is \emph{self-similar} if it is invariant under the right action of $X^*$ on $S_K$ by shift maps. 
%A closed, self-similar subgroup of $S_K$ is a tree-subshift with alphabet $\Sym(X)\times G/K$. 
%We will say that a self-similar subgroup $G$ of $S_K$ is \emph{symbolically regular branch} if it contains some normal finite-index subgroup $H$ such that $X\ast H\leq H\leq G$, where $X\ast H=\{x\ast h\mid x\in X, h\in H\}$ and $x\ast h\in S_K$ is the element $f$ of $S_K$ such that $f_x=h$ and $f_u$ is trivial for all $u\notin xX^*$.
%In that case, $G$ will be said to \emph{branch symbolically over $H$}.
%\end{definition}

Given any self-similar group $G\leq \Aut(X^*)$ that branches over some finite-index normal subgroup $K$, there is a natural embedding
 $$\theta_K:G\hookrightarrow S_K,\quad g\mapsto (g(u), (g_{u}K)_{u\in X^*})$$
 where $g_{u}K$ is,  for $u\in X^*$, the image modulo $K$ of the $u$-shift $g_{u}\in G$ of $g\in G$. 
By the definition of multiplication in $S_K$, the map $\theta_K$ is a homomorphism. 
Recalling that we have identified $S_K=\Aut(X^*)\ltimes (G/K)^{X^*}$ with $(\Sym(X)\times G/K)^{X^*}$, the element $\theta_K(g)$ can be identified with $(g_u(\St(1)\cap K))_{u\in X^*}$, because of the wreath product structure we have given to $(\Sym(X)\times G/K)^{X^*}$ and because we are viewing $G/(\St(1)\cap K)$ as a subgroup of $G/K\times G/\St_G(1)\leq G/K\times \Sym(X)$.

\begin{notation}
	We will identify $G$ with its image $\theta_K(G)\leq S_K$ whenever this is convenient.
	Abusing notation, $\Triv_G^K(n)$ will denote $\Triv^K_{\Theta_K(G)}(n)=\theta_K^{-1}(\Triv^K(n))=\St_G(n)\cap X^{n-1}\ast K$.
%	The preimage $\theta_K^{-1}(\Triv^K(n))$ of $\Triv^K(n)$ is denoted by $\Triv^K_G(n)$.
\end{notation}

\begin{definition}
The closure $\widetilde{G_K}$ of $G$ in $S_K$ is a profinite group that is self-similar over $\Sym(X)\times G/K$ and branches symbolically over the closure of $K$ in $S_K$. 
It is called the \emph{$K$-symbolic completion} of $G$. 
It can also be seen as the inverse limit 
$$\widetilde{G_K}=\varprojlim_{n\in\NN}G/\Triv_G^K(n)$$ 
of the inverse system  of groups $\{G/\Triv_G^K(n) : n\geq 1\in \NN \}$ and canonical homomorphisms
 $\{t_{m,n}:G/\Triv_G^K(m)\rightarrow G/\Triv_G^K(n) : m\geq n \geq 1 \in \NN\}$.% where 
%$$\Triv^K_G(n)=\theta_K^{-1}(\Triv^K(n))=\St_G(n)\cap X^{n-1}\ast K.$$
\end{definition}
 
\begin{definition}
Given a group $G\leq \Aut(X^*)$ one can also consider,
for a vertex $v\in X^*$, the rigid stabiliser of $v$ in $G$:
$$\rst_G(v)=\{g\in G: g(u)=\mathrm{id}\in\Sym(X) \text{ for all } X^*\setminus vX^*\}$$
the subgroup that fixes all vertices outside the subtree $vX^*$ with root $v$. 
Note that, if $vX^*\cap uX^*=\emptyset$, then $\rst_G(v)$ and $\rst_G(u)$ commute. 
In particular, the subgroup generated by all $\rst_G(v)$ for $v\in X^n$ is the direct product
$$\rst_G(n)=\prod_{v\in X^n}\rst_G(v)$$
and called the \emph{rigid level stabiliser} of level $n$ in $G$.  

If $G$ is regular branch  over $K$, then $X\ast K \leq \rst_G(1)$ has finite index in $G$. 
Continuing inductively,  $X^n\ast K\leq \rst_G(n)$ all have finite index in $G$, for $n\geq 1$. 
This last condition, along with a transitive action on $X^n$ for each $n$, ensures that $G$ is a \emph{branch group}. 
The completion $\varprojlim_n G/\rst_G(n)$  is called the \emph{branch completion} of $G$. 

\end{definition}

In general, since $\Triv^K_G(n)=\St_G(n)\cap X^{n-1}\ast K\leq \rst_G(n-1)$, the $K$-symbolic completion of $G$ maps onto the branch completion. 
For a certain choice of $K$, these two completions actually coincide:

The subgroup $M:=\bigcap_{u\in X^*}(\rst_G(u))_u$ is the unique maximal branching subgroup of $G$. 
If $G$ is transitive on all levels of $X^*$, then, according to \cite[Corollary 1.6]{bartholdi-s-z:congruence}, there exists $n\geq0$ such that $X^m\ast M\geq \rst_G(m+n)$ for all $m\geq0$. 
In particular, 
$$\Triv^M_G(m+1)=\St_G(m+1)\cap X^m\ast M\geq \St_G(m+1)\cap \rst_G(m+n)\geq \rst_G(m+n+1).$$
Thus $\{\Triv^M_G(n): n\geq 1\in\NN\}$ and $\{\rst_G(n):n\geq 1\in\NN\}$ generate the same profinite topology, so the $M$-symbolic and branch completions of $G$ coincide. 

%[BSZ]Corollary 1.6: Let Γ be a self-similar, level-transitive, regular branch group.
%Then there exists n ≥ 0 such that X m ∗ Γ # ≥ Rist Γ (m + n) for all m ≥ 0.

\begin{remark}
	In the remainder, we shall always assume that $G\leq \Aut(X^*)$ is a self-similar, level-transitive, regular branch group, so that we can identify the $M$-symbolic with the branch completion of $G$, and omit the $M$ notation, writing $\Triv_G(n)$ instead of $\Triv^M_G(n)$. 
\end{remark}

\subsection{Completions and kernels}\label{subsec:kernels}

Let $G\leq \Aut(X^*)$ be a self-similar, level-transitive regular branch group with maximal branching subgroup $K$, so that there exists $k\in \NN$ such that $\Triv_G(n)\geq \rst_G(n+k)\geq\Triv_G(n+k)$ for every $n\in \NN$.
For every $m\geq n\geq 1\in \NN $, the following square of canonical morphisms commutes:
\[
\begin{tikzcd}
	G/\Triv_G(m) \arrow[r, "\psi_m"] \arrow[d, "t_{m,n}"] & G/\St_G(m) \arrow[d, "s_{m,n}"] \\
	G/\Triv_G(n) \arrow[r, "\psi_n"] & G/\St_G(n)               
\end{tikzcd}
\]
This gives a unique morphism $\psi: \widetilde{G}=\varprojlim_n G/\Triv_G(n) \rightarrow \overline{G}=\varprojlim_n G/\St_G(n)$ defined by $\psi((g_n)_n)=(\psi_n(g_n))_n$, whose kernel is 
$$\ker \psi=\varprojlim_n \ker\psi_n=\varprojlim_n\St_G(n)/\Triv_G(n)$$
the inverse limit of the inverse system whose maps $r_{m,n}:\St_G(m)/\Triv_G(m)\rightarrow  \St_G(n)/\Triv_G(n)$  are simply the restrictions of the $t_{m,n}$ for all  $m\geq n\geq 1\in \NN$.

\begin{definition}
	With notation as above, $\ker\psi$ is the \emph{rigid (or symbolic) kernel} of $G$.
\end{definition}

%The maps $r_{m,n}:\St_G(m)/\Triv_G(m)\rightarrow  \St_G(n)/\Triv_G(n)$ in the inverse system are simply the restrictions of the $t_{m,n}:G/\Triv_G(m)\rightarrow G/\Triv_G(n)$.

Since $G$ is residually finite, it embeds into its profinite completion $\widehat{G}=\varprojlim_{N\unlhd_f G} G/N$, the inverse limit of the inverse system $\{G/N : N\unlhd_f G\}$ of finite quotients of $G$, and canonical maps $q_{M,N}:G/M\rightarrow G/N$ for $M\leq N \unlhd_f G$. 

Since each $\Triv_G(n)$ and $\St_G(n)$ are normal subgroups of finite index in $G$, restricting to $\widehat{G}$ the projection maps to, respectively,  $\prod_{n}G/\Triv_G(n)$ and $\prod_{n}G/\St_G(n)$ gives surjective morphisms $\phi:\widehat{G}\rightarrow \widetilde{G}$ and $\theta: \widehat{G}\rightarrow \overline{G}$.
Moreover, $\theta= \psi\circ\phi$. 

This gives two further kernels: the \emph{branch kernel} $\ker \phi =\varprojlim_{N\unlhd_f \Triv_G(n), n}\Triv_G(n)/N$ and the \emph{congruence kernel} $\ker\psi\circ\phi=\varprojlim_{N\unlhd_f\St_G(n), n}\St_G(n)/N$.

These completions and their kernels are important to understand the structure of groups acting on rooted trees, as they give insight into their finite quotients. 
By analogy with arithmetic groups, groups for which the congruence kernel is trivial are said to have the \emph{congruence subgroup property}.

Examples include many of the first discovered self-similar groups (Grigorchuk, Gupta--Sidki, etc.,). 
The first examples of groups with non-trivial congruence kernels were constructed by Pervova in \cite{pervova:completions}, where she also showed that they have trivial rigid kernel. 
It was not until \cite{bartholdi-s-z:congruence}, where a systematic study of the congruence subgroup property for branch groups was undertaken, that an example (the Hanoi towers group) was shown to have non-trivial branch and rigid kernel. 
Further examples followed in \cite{skipper:congruence}. 
The computations to find the rigid kernel were technical in these cases. 
The advantage of looking at tree-shifts and portraits is that it clarifies what these kernels are:

The map $\psi$ is precisely the restriction to $\widetilde{G}$ of the canonical epimorphism $S_M=\Aut(X^*)\ltimes (G/M)^{X^*} \twoheadrightarrow \Aut(X^*)$.
The kernel of $\psi$  therefore  consists of all tree portraits in $\widetilde{G}$ that have trivial $\Sym(X)$ part in every vertex.

\subsection{Finitely constrained groups}
 We briefly return to the general setting of shifts of $A^{X^*}$ (closed subsets that are invariant under shift maps), where $A$ is any set. 
 A \emph{pattern} is a function from $X^{<n}$ to $A$ for some $n\geq 1$, which is the \emph{size} of the pattern.
 We say that a pattern of size $n\geq 1$ \emph{appears} in some $f\in A^{X^*}$ if there exists $u\in X^*$ such that $f_u|_{X^{<n}}$ is the specified pattern.
 We will also say that this is \emph{the pattern of size $n$ at $u$ of $f$. }
 
 It is well-known in symbolic dynamics that any shift can be defined by declaring some collection of \emph{forbidden patterns} which do not appear in any element of the shift.
  If a finite collection of forbidden patterns suffices for this definition, the shift is called a \emph{shift of finite type}.
  
In the case where $A$ is a group, one can analogously define and study \emph{groups of finite type} or \emph{finitely constrained groups}: subgroups of $A^{X^*}$ that are simultaneously shifts of finite type. 
These groups are characterised in the following theorem \cite{penland-sunic:kitchens}.
%Here, $\Triv_G(n)=\{f\in G\leq A^{X^*}: f(u)=\mathrm{id}\in A \text{ for all } u\in X^{< n}\}$. 

\begin{theorem}[\cite{penland-sunic:kitchens}]\label{thm:ps_fincons_regbranch}
	 Let $\Gamma\leq A^{X^*}$ where $A$ is a group that acts on $X$, and $n\geq 1$. The following are equivalent:
\begin{enumerate}
 \item $\Gamma$ is a finitely constrained group defined by patterns of size $n$.
 \item $\Gamma$ is closed, self-similar over $A$, and branches symbolically over  $\Triv_{\Gamma}(n-1)$.
 \item  $\Gamma$ is the closure of a group $G\leq A^{X^*}$ that is self-similar over $A$ and branches symbolically over $\Triv_G(n-1)$.
\end{enumerate}
\end{theorem}

 Observe that if $G\leq \Aut(X^*)$ is a self-similar, regular branch group, with maximal branching subgroup $K$, then 
 $$\Triv^K_G(2)=\St_G(2)\cap X\ast K = X\ast (\St_G(1)\cap K)=X\ast \Triv^K_G(1).$$
 In other words, $G$ branches over $\Triv^K_G(1)$, so, after making the identifications explained in the previous section, we obtain from Theorem \ref{thm:ps_fincons_regbranch}:
 
 \begin{corollary}\label{cor:symb_compl_fincons_patterns2}
 	If $G\leq\Aut(X^*)$ is a self-similar, regular branch group, with maximal branching subgroup $K$, then the symbolic (branch) completion $\widetilde{G_K}$ is a finitely constrained group defined by patterns of size 2. 
 \end{corollary} 

Moreover, if $G$ is a self-similar, regular branch group, with maximal branching subgroup $K$ then the kernel $R_K$ of the canonical `forgetful' map $\Psi: S_K=\Aut(X^*)\ltimes (G/K)^{X^*} \twoheadrightarrow \Aut(X^*)$ is precisely $R_K=\mathrm{id}\ltimes (G/K)^{X^*}$, which is evidently a finitely constrained group defined by patterns of size 1. 
The rigid kernel of $G$ is therefore $\widetilde{G_K}\cap R_K$, the intersection of two closed groups that are self-similar over $\Sym(X)\times G/K$, and it  branches symbolically over $\Triv^K_{R_K\cap\widetilde{G_K}}(1)$.
Applying Theorem \ref{thm:ps_fincons_regbranch} we conclude:

\begin{corollary}\label{cor:ker_fs_2}
	If $G\leq\Aut(X^*)$ is a self-similar, regular branch group, with maximal branching subgroup $K$, then the rigid kernel $\ker:\widetilde{G_K}\rightarrow \cl{G}$ is a finitely constrained group defined by patterns of size 2. 
\end{corollary}

Corollary \ref{cor:symb_compl_fincons_patterns2} also  has consequences for the structure of $\cl{G}$ as a tree shift. 
It is natural to wonder whether it is also a finitely constrained group. 
A priori, it is only a \emph{sofic tree shift}.
For our purposes, (see \cite[Section 2]{cech-c-f-s:cellautomata} for equivalent definitions),  given two finite alphabets $A, B$, a sofic tree shift in $B^{X^*}$  is the image of a tree shift of finite type in $A^{X^*}$  under a  continuous map that commutes with the shift action of $X^*$. 

The canonical  map $\Psi:S_K=\Aut(X^*)\ltimes(G/K)^{X^*} \rightarrow \Aut(X^*), (f,h)\mapsto f$ is indeed continuous and it commutes with the shift action of $X^*$ on $S_K$ and $\Aut(X^*)$ because 
$$\Psi((f,h)_u)(v)=f(uv)=\Psi(f,h)_u(v)\in \Sym(X)$$
holds for all $(f,h)\in S_K$ and $u,v\in X^*$. 
The map $\psi:\widetilde{G}\rightarrow \overline{G}$ is the restriction of $\Psi$ to the tree shift of finite type $\widetilde{G}$, so $\overline{G}$ is indeed a sofic tree shift. 

The following theorem of Penland and \v{S}uni\'{c} shows that, in many cases, groups that are sofic tree shifts are already finitely constrained groups.

\begin{theorem}[Theorem A of \cite{penland-sunic:kitchens}]
	Let $A$ be a finite group acting on $X$ and $G\leq A^{X^*}$.
	If the normalizer of $G$ in $A^{X^*}$ contains a level-transitive subgroup that is self-replicating over $A$, then $G$ is a sofic tree shift group if and only if $G$ is a finitely constrained group. 
\end{theorem}

This together with Corollary \ref{cor:symb_compl_fincons_patterns2}  allows us to conclude:
\begin{corollary}\label{cor:cong_completion_finitely_constrained}
	Let $G\leq \Aut(X^*)$ be a self-replicating, level-transitive, regular branch group. 
	Then the closure $\cl{G}$ of $G$ in $\Aut(X^*)$ is a finitely constrained group. 
\end{corollary}

Note that the above says nothing about the size of defining patterns of $\overline{G}$, just that they are bounded.

\begin{theorem}[Theorem \ref{t:criterion}]%\label{t:criterion}
	Let $G \leq \Aut(X^*)$ be a self-similar, level-transitive, regular branch group, with maximal branching subgroup $K$. 
 	The following are equivalent:
	\begin{enumerate}
		\item $G$ has  trivial rigid kernel;
		\item $G$ branches over some level stabiliser;
		\item $K\geq \St_G(n)$ for some $n\geq 1$.
	\end{enumerate}
	
%		\item $G$ branches over its second level stabiliser;
	If $G$ is in addition self-replicating, then $\cl{G}$ is a  finitely constrained group and the above items are also equivalent to 
	
	\begin{enumerate}
		\setcounter{enumi}{3}
		\item $ |X\ast G:\St_G(1)] =  |X\ast \cl{G}:\cl{\St_G(1)}]$.
	\end{enumerate}
\end{theorem}

%As will be apparent from the proof below, the first three items are equivalent without requiring $G$ to be self-replicating, only self-similar. 

\begin{proof}
	Let $G\leq \Aut(X^*)$ be a self-similar, level-transitive, regular branch group, with maximal branching subgroup $K$. 
	
	(i)$\Leftrightarrow$ (ii). The rigid kernel of $G$ is trivial if and only if $\{\Triv_G(n):n\in\NN\}$ and $\{\St_G(n):n\in \NN\} $ generate the same topology on $G$, which holds if and only if for every $n\in \NN$ there is some $m\in\NN$ such that $\Triv_G(n)\geq \St_G(m)$. 
	This is equivalent to the existence of some $s\in\NN$ such that $X\ast \St_G(s)=\St_G(s+1)$. 
	To wit, if $\Triv_G(n)\geq \St_G(m)$ then, since $G$ branches over $\Triv_G(1)$ and $K$ is the maximal branching subgroup, we have $\St_G(m)\leq K$ and thus $X\ast \St_G(m)\leq \St(m+1)\cap (X\ast K)\leq \St_G(m+1)$, so $X\ast\St_G(m)=\St_G(m+1)$. 
	Conversely, if $X\ast \St_G(s)=\St_G(s+1)$ then $G$ branches over $\St_G(s)$ and therefore $\St_G(s)\leq K$, so $\St_G(s)=\St_K(s)$ and 
	$\St_G(n+s)= X^n\ast \St_K(s)\leq X^n\ast \St_K(1)=\Triv_G(n+1)$ for every $n\in \NN$.
	
%	(i)$\Leftrightarrow $ (iii).	The rigid kernel is trivial if and only if for every $n\geq 1$ there exists some $m\geq n$ such that 
%	$\Triv_G(n)=X^{n-1}\ast(K\cap \St_G(1))\geq \St_G(m).$%=G\cap (X^{n-1}\ast \St_G(m-n+1)).$$
%	For $n=1$, this implies that $K\geq K\cap \St_G(1)\geq \St_G(m)$ for some $m\geq 1$. 
%	Conversely, if $K\geq \St_G(s)$ for some $s\geq 1$ then
%	$\Triv_G(n)=X^{n-1}\ast(K\cap \St_G(1))\geq X^{n-1}\ast \St_G(s)=\St_G(s+n-1)$ for every $n\geq 1$.
	
	(ii)$\Leftrightarrow$ (iii). If $K$ is the maximal branching subgroup, then $K\geq \St_G(n)$ for some $n\geq 1$ if and only if $\St_G(n+1)\geq X\ast \St_G(n)$.
	
	If $G$ is also assumed to be self-replicating, then Corollary \ref{cor:cong_completion_finitely_constrained} says that $\cl{G}$ is a finitely constrained group.
	In particular, by Theorem \ref{thm:ps_fincons_regbranch}, there exists some $n\geq 0$ such that
	$\cl{\St_G(n+1)}=X\ast \cl{\St_G(n)}$. 
	This implies that 
	\begin{align*}
		|X\ast G:X\ast \St_G(n)|&=|X\ast \cl{G}:X\ast \cl{\St_G(n)}|=|X\ast \cl{G}:\cl{\St_G(n+1)}|\\
		&=|X\ast \cl{G}:\cl{\St_G(1)}||\cl{\St_G(1)}:\cl{\St_G(n+1)}|\\
		&=|X\ast \cl{G}: \cl{\St_G(1)}||\St_G(1):\St_G(n+1)|.
	\end{align*} 
	At the same time, 
	
	$$|X\ast G:X\ast \St_G(n)|=\frac{|X\ast G:\St_G(n+1)|}{|X\ast \St_G(n):\St_G(n+1)|}=\frac{|X\ast G:\St_G(1)|\cdot |\St_G(n):\St_G(n+1)|}{|X\ast \St_G(n):\St_G(n+1)|}.$$
	
	Putting these together,  $|X\ast G:\St_G(1)|=|X\ast \cl{G}: \cl{\St_G(1)}|\cdot |X\ast \St_G(n):\St_G(n+1)|$, therefore  $|X\ast G:\St_G(1)|=|X\ast \cl{G}: \cl{\St_G(1)}|$ if and only if $X\ast \St_G(n)\leq \St_G(n)$. 
	This shows that items (ii) and (iv) are equivalent. 
\end{proof}

Note that the above result is not ``quantitative''.
 Namely, when the two indices in condition (iv) are not equal, we have no information on the rigid kernel other than that it is non-trivial.

% Let us now turn to determining the structure of the rigid kernel.  

\subsection{The rigid kernel}\label{subsec:rigidkernels}
For the rest of this section, fix $G\leq\Aut(X^*)$, a self-similar, level-transitive, regular branch group, with maximal branching subgroup $K$.
We will drop the $K$ notation from now on, so $\widetilde{G}=\widetilde{G_K}$ and $R=R_K$, etc. 

As stated above, the rigid kernel 
%$$\ker \left(\widetilde{G}=\varprojlim_n G/\Triv_G(n) \rightarrow \overline{G}\varprojlim_n G/\St_G(n)\right) = \varprojlim_n \St_G(n)/\Triv_G(n)$$
$\ker (\widetilde{G}\rightarrow \cl{G})$ is $\widetilde{G}\cap R$
where $R= \mathrm{id}\ltimes (G/K)^{X^*}$ and it is a finitely constrained group defined by patterns of size 2, by Corollary \ref{cor:ker_fs_2}.

Notice in particular that the rigid kernel inherits from $R$ any property of $G/K$ that is inherited by subgroups of Cartesian products (for instance, and this will be important for our examples, being elementary abelian).

The patterns of size 2 defining $\widetilde{G}\cap R$ are in $\St_G(2)/\Triv_G(2)\cong\St_G(2)(X\ast K)/(X\ast K)$. 
Once a root pattern is fixed, the patterns of size 2 at each $x\in X$ must be chosen so that their value at $x$ (which must be in $\{\mathrm{id}\}\times G/K$) matches the value at $x$ of the root pattern of size 2. 
This means that the pattern of size 2 at $x\in X$ is in a coset of $(\St_G(2)\cap \Triv_G(1))/\Triv_G(2)$ inside $\St_G(2)/\Triv_G(2)$; or, equivalently, in a coset of 
$(\St_G(2)(X\ast K)\cap K)/(X\ast K)$ in $\St_G(2)(X\ast K)/(X\ast K)$.
Note that  $K/(X\ast K)$ is not trivial in general so there may still be non-trivial elements in $(\St_G(2)(X\ast K)\cap K)/(X\ast K)$.

%So the patterns of size 2 at each $x\in X$ must be chosen from $(\St_G(2)\cap \Triv_G(1))/\Triv_G(2)=(\St_G(2)\cap K)/\Triv_G(2)$. 
%Notice that since $\Triv_G(1)/\Triv_G(2)=K/(X\ast K)$ is not trivial in general, there may still be non-trivial elements in $(\St_G(2)\cap \Triv_G(1))/\Triv_G(2)$,

The same argument holds for all $v\in X^*\setminus\{\epsilon\}$%. 
%Since $R$ acts trivially on $\Aut(X^*)$
, so  $\widetilde{G}\cap R$  %is a closed subgroup of
%$$\Gamma:=\frac{\St_G(2)}{\Triv_G(2)}\times \prod_{X^*\setminus\{\epsilon\}} \frac{\St_G(2)\cap K}{\Triv_G(2)}.$$
is the inverse limit of the  extensions 
$$X^{n-2}\ast\frac{\St_G(2)\cap K}{\Triv_G(2)}\rightarrow \St_G(n)/\Triv_G(n)\rightarrow \St_G(n-1)/\Triv_G(n-1) \qquad n\geq 3.$$

%\begin{lemma}
%	Let $G$ be a self-similar level-transitive regular branch group with maximal branching subgroup $K$. 
%	The rigid kernel of $G$ is trivial if and only if $K\geq\St_G(n)$ for some $n$.
%\end{lemma}
%\begin{proof}
%	The rigid kernel is trivial if and only if for every $n\geq 1$ there exists some $m\geq n$ such that 
%	$$\Triv_G(n)=X^{n-1}\ast(K\cap \St_G(1))\geq \St_G(m)=G\cap (X^{n-1}\ast \St_G(m-n+1))$$
%	which holds if and only if $K\cap\St_G(1)\geq \St_G(m-n+1)$. 
%\end{proof}

%The following useful remark is a straightforward exercise in the definitions of inverse limit.
\begin{lemma}\label{lem:pass_to_surjective_inverse_system}
	Let $X$ be the inverse limit of the inverse system $\{X_i, \varphi_{j,i}:X_j\rightarrow X_i \mid j\geq i\geq 1\in \NN\}$ of finite groups and homomorphisms. 
	Then $X$ is also the inverse limit of the surjective inverse system $\{Y_i, \psi_{j,i} \mid j\geq i\geq 1\}$ where $Y_i=\bigcap_{j\geq i}\im \varphi_{j,i}$ and $\psi_{j,i}=\varphi|_{Y_j}$.
\end{lemma} 
\begin{proof}
A straightforward exercise in the definitions of inverse limit.
	\end{proof}

\begin{theorem}\label{thm:finite_rigid_kernel_inverse_system_criterion}
	Let $G$ be a self-similar, level-transitive regular branch group with maximal branching subgroup $K$. 
	\begin{enumerate}
	\item The rigid kernel of $G$ is finite  if and only if %there exist $i\geq 1$ such that $\St_G(n-i+1)\Triv_G(n-i+1)\cap K = \Triv_G{n-i+1}$ for every $n\geq i$.
	$$\bigcap_{n\geq m}(\St_G(n)\Triv_G(m))\cap K \leq \Triv_G(m) \text{ for every } m\geq 1.$$ 
	
	\item If the system $\{\St_G(m)/\Triv_G(m), r_{n,m}\}$ is surjective, then the rigid kernel is finite if and only if $\St_G(2)\cap K\leq \Triv_G(2)$, in which case the rigid kernel is isomorphic to $\St_G(1)/\Triv_G(1)$. 	
	\end{enumerate}
\end{theorem}
\begin{proof}
	Using Lemma \ref{lem:pass_to_surjective_inverse_system}, we can replace each $\St_G(i)/\Triv_G(i)$ by $H_i:=\bigcap_{k\geq i}\St_G(k)\Triv_G(i)/\Triv_G(i)$ and each $r_{j,i}$ by its restriction to $H_j$ (which we still denote by $r_{j,i}$). 
	The rigid kernel is finite if and only if there exists $i\geq 1$ such that $r_{j,i}$ is an isomorphism for each $j\geq i$, which occurs if and only if the kernel of $r_{j,i}$ is trivial. 
	This is equivalent to $\Triv_G(i)\cap \bigcap_{k\geq j}\St_G(k)\Triv_G(j)=\Triv_G(j)$, which occurs if and only if the left-hand side is contained in the right-hand side, as the other containment holds always. 
	Now, 
	\begin{multline*}
		\Triv_G(i)\cap \bigcap_{k\geq j}\St_G(k)\Triv_G(j) \\
		= X^{i-1}\ast (K\cap \St_G(1))\cap \bigcap_{k\geq j}\left(G\cap X^{i-1}\ast (\St_G(k-i+1)\Triv_G(j-i+1))\right) \\
		=X^{i-1}\ast (K\cap \bigcap_{k\geq j}\St_G(k-i+1)\Triv_G(j-i+1)) 
	\end{multline*}
	and similarly, $\Triv_G(j)=X^{i-1}\ast \Triv_G(j-i+1)$. 
	So the rigid kernel is finite if and only if, there exists $i\geq 1$ such that 
	$$X^{i-1}\ast (K\cap \bigcap_{k\geq j}\St_G(k-i+1)\Triv_G(j-i+1)) \leq X^{i-1}\ast \Triv_G(j-i+1)$$
	for every $j\geq i$. 
	But this occurs exactly when $$(K\cap \bigcap_{k\geq j}\St_G(k-i+1)\Triv_G(j-i+1)) \leq\Triv_G(j-i+1).$$
	This makes $j-i$ the only variable, so that the rigid kernel is finite if and only if 
	\begin{equation}\label{eq:finite_rigid_kernel}
		(K\cap \bigcap_{n\geq m}\St_G(n)\Triv_G(m)) \leq\Triv_G(m) \text{ for every } m\geq 1.
	\end{equation}
	
	If the system $\{\St_G(m)/\Triv_G(m), r_{n,m}\}$ is surjective, then $\St_G(n)\Triv_G(m)=\St_G(m)$ for every $n\geq m\geq 1$, so \eqref{eq:finite_rigid_kernel} becomes $K\cap \St_G(m)\leq \Triv_G(m)$ for every $m\geq 1$. 
	In particular, $K\cap \St_G(2)=\Triv_G(2)$. 
	Now, suppose that the inverse system is surjective and that $K\cap \St_G(2)=\Triv_G(2)$. 
	Then, for each $m\geq 2$, the kernel of $r_{m,m-1}$ is %:\St_G(m)/\Triv_G(m)\rightarrow \St_G(m-1)/\Triv_G(m-1)$ is
	$$\frac{\St_G(m)\cap \Triv_G(m-1)}{\Triv_G(m)}=\frac{X^{m-2}\ast (\St_G(2)\cap K)}{\Triv_G(m)}=\frac{X^{m-2}\ast \Triv_G(2)}{\Triv_G(m)} =\frac{\Triv_G(m)}{\Triv_G(m)}$$
	where the second-to-last equality holds because $K\cap \St_G(2)=\Triv_G(2)$.
	This means, that $r_{m,m-1}$ is an isomorphism, for each $m\geq 2$, and therefore the rigid kernel is isomorphic to $\im(r_{m,1})=\St_G(1)/\Triv_G(1)$, which is finite, as required. 	
\end{proof}

\begin{lemma}\label{lem:surjective_system_iff_G2K=G1K}
	The inverse system $\{\St_G(n)/\Triv_G(n), r_{m,n} : m\geq n\geq 1\in \NN \}$ is surjective if and only if $\St_G(2)K \geq \St_G(1)$.
\end{lemma}
\begin{proof}
	For every $m\geq n\geq 1$ the image of $r_{m,n}$ is $\St_G(m)\Triv_G(n)/\Triv_G(n)$.
	Observe that 
	\begin{multline}\label{eq:surjective system}
		\St_G(n+1)\Triv_G(n) = (G\cap X^{n-1}*\St_G(2))(X^{n-1}\ast \Triv_G(1)) \\
		=G\cap X^{n-1}\ast(\St_G(2)\Triv_G(1))=G\cap X^{n-1} \ast (\St_G(1)\cap\St_G(2)K) \\
		\leq G\cap X^{n-1}\ast \St_G(1)=\St_G(n)
	\end{multline}
	where the third equality follows from $\Triv_G(1)=\St_G(1)\cap K$. 
	If the inverse system is surjective, then, in particular, $\St_G(2)\Triv_G(1)=\St_G(1)\cap \St_G(2)K=\St_G(1)$, so $\St_G(2)K\geq \St_G(1)$.
	Conversely, if $\St_G(2)K\geq \St_G(1)$ then we obtain equalities in \eqref{eq:surjective system}. % we obtain $\St_G(n+1)\Triv_G(n)=G\cap X^{n-1}\ast(\St_G(1))=\St_G(n)$. 
	Inductively, this implies that $\St_G(m)\Triv_G(n)=\St_G(n)$ for every $m\geq n\geq 1$. 
\end{proof}

\begin{theorem}[Theorem \ref{thm:rigid_kernel_surjective_system_direct_prod}]
	Let $G$ be a self-similar, level-transitive regular branch group with maximal branching subgroup $K$. 
	Suppose that $\St_G(2)K\geq \St_G(1)$. 
	\begin{enumerate}
		\item The rigid kernel of $G$ is trivial if and only if $\St_G(2)=\Triv_G(2)$. 
%		\item It is infinite if and only if $\St_G(2)\cap K \gneq \Triv_G(2)$. 
		\item Suppose moreover that $G/K$ is in a class of groups that is closed under subgroups, quotients and direct products, in which all short exact sequences that are in the class split as direct products (for example, elementary abelian groups).
		Then the rigid kernel is  $$\Gamma=\frac{\St_G(2)}{\Triv_G(2)}\times \prod_{X^*\setminus\{\epsilon\}} \frac{\St_G(2)\cap K}{\Triv_G(2)}$$
		and it is infinite if and only if $\dfrac{\St_G(2)\cap K}{\Triv_G(2)}$ is non-trivial.
	\end{enumerate}
\end{theorem}
\begin{proof}
	According to Lemma \ref{lem:surjective_system_iff_G2K=G1K}, the inverse system $\{\St_G(i)/\Triv_G(i), r_{j,i}\}$ is surjective. 
%	By theorem \ref{}, the rigid kernel is finite if and only if $\St_G(2)\cap K=\Triv_G(2)
	\begin{enumerate}
		\item The rigid kernel is trivial if and only if $\St_G(n)=\Triv_G(n)$ for every $n\geq 1$. 
	Conversely, suppose that $\St_G(2)=\Triv_G(2)$. 
	Then $\St_G(2)\cap K=\Triv_G(2)$, so Theorem \ref{thm:finite_rigid_kernel_inverse_system_criterion} implies that the rigid kernel is isomorphic to $\St_G(1)/\Triv_G(1)$, which is trivial because it is the image of the trivial quotient $\St_G(2)/\Triv_G(2)$.
%		\item This is simply Theorem \ref{}.
		\item For every $n\geq 2$, we have the following short exact sequence:
			$$ 1 \rightarrow \frac{\St_G(n)\cap \Triv_G(n-1)}{\Triv_G(n)} \rightarrow \frac{\St_G(n)}{\Triv_G(n)}\rightarrow \frac{\St_G(n-1)}{\Triv_G(n-1)}\rightarrow 1$$
			where $\frac{\St_G(n)\cap \Triv_G(n-1)}{\Triv_G(n)}=X^{n-2}\ast\left(\frac{\St_G(2)\cap K}{\Triv_G(2)}\right)$. 
			%Since $\St_G(1)/\Triv_G(1)\cong \St_G(1)K/K \leq G/K$, the assumption on $G/K$ guarantees that the above short exact sequence is split
			Since $\St_G(n)/\Triv_G(n)\leq \prod_{X^{<n}}G/K$, the assumption on $G/K$ implies that the above short exact sequence splits as  a direct product. 
			Inductively, this means that 
			$$\frac{\St_G(n)}{\Triv_G(n)}=\frac{\St_G(2)}{\Triv_G(2)}\times \prod_{\epsilon\neq x\in X^{<n}}x\ast \frac{\St_G(2)\cap K}{\Triv_G(2)}$$
			for every $n\geq 2$. 
			Taking the inverse limit gives the result. 
	\end{enumerate} 
\end{proof}

For all self-similar examples in the literature where the congruence subgroup problem has been considered, the inverse system  $\{\St_G(n)/\Triv_G(n), r_{m,n} : m\geq n\geq 1\in \NN \}$  turns out not to be surjective, except for the Hanoi Towers group $H^{(3)}$, as will be seen in the next section. 
In all the other cases, the rigid kernel is trivial anyway because $K\geq \St_G(n)$ for some $n$. 
It would be interesting to have examples of self-similar branch groups with non-trivial rigid kernel and such that the inverse system  $\{\St_G(n)/\Triv_G(n), r_{m,n} : m\geq n\geq 1\in \NN \}$ is not surjective.

\section{An infinite family of Hanoi-like groups}\label{sec:3}

We analyse, for odd $d$, the group $D=\langle a_0, \dots, a_{d-1} \rangle$ and prove the items in Theorem \ref{t:list}.

\begin{notation}
%	For $i\in \{0, \dots, d-1\}$, and $g\in \Aut(X^*)$ define $\delta_i(g):=i\ast g$. 
		All arithmetic operations on the indices are done modulo $d$.
\end{notation}

%\subsection{The dihedral group $D(d)$}
\subsection{Element decomposition and the contracting property}

The following are straightforward to verify. 

\begin{lemma}\label{lem:dihedral_relations}
\begin{alignat*}{2}
 \mu_i(x)  &= 2i-x \qquad  \rho^i(x) &&= i+x \\
 \mu_i\mu_j & = \rho^{2(i-j)} \qquad  \rho^i \rho^j &&= \rho^{i+j} \\
 \mu_i\rho^j & = \mu_{i-j/2}  \qquad \rho^j \mu_i &&= \mu_{i+j/2} \\
 \mu_i\rho^j\mu_i & = \rho^{-j}  \qquad \rho^{-j} \mu_i \rho^j &&= \mu_{i-j} \\
\end{alignat*}
\begin{align*}
 \mu_{i_m} \mu_{i_{m-1}} \dots \mu_{i_2} \mu_{i_1} &= 
  \begin{cases}
     \rho^{2(i_m-i_{m-1}+i_{m-2}-i_{m-3} + \dots + i_2 - i_1)}, 
      & m \textup{ even} \\
     \mu_{i_m-i_{m-1}+i_{m-2}-i_{m-3} + \dots -i_2 + i_1}, & m \textup{ odd}  
  \end{cases} \\
 \mu_{i_m} \mu_{i_{m-1}} \dots \mu_{i_2} \mu_{i_1}(x) &= 
  \begin{cases}
     2(i_m-i_{m-1}+i_{m-2}-i_{m-3} + \dots + i_2 - i_1)+x, 
      & m \textup{ even} \\
     2(i_m-i_{m-1}+i_{m-2}-i_{m-3} + \dots -i_2 + i_1)-x, & m \textup{ odd}  
  \end{cases} \\  
\end{align*}
\end{lemma}

\begin{lemma}\label{lem:section_decompos_elts_hanoihedral}
Let $g = a_{i_m} a_{i_{m-1}} \dots a_{i_2} a_{i_1}$. The root permutation of $g$ is 
\[
 g_{(\epsilon)} = \mu_{i_m} \mu_{i_{m-1}} \dots \mu_{i_2} \mu_{i_1} = 
  \begin{cases}
     \rho^{2(i_m-i_{m-1}+i_{m-2}-i_{m-3} + \dots + i_2 - i_1)}, 
      & m \textup{ even} \\
     \mu_{i_m-i_{m-1}+i_{m-2}-i_{m-3} + \dots -i_2 + i_1}, & m \textup{ odd}  
  \end{cases}
\]
The following table gives the position to which each $a_{i_j}$ is contributed in the first level decomposition of $g$:
\[
 \begin{array}{cl}
 a_{i_j} & \textup{contributed to position} \\
 \hline
 a_{i_1} & i_1 \\ 
 a_{i_2} & 2i_1 -i_2 \\ 
 a_{i_3} & 2i_1 -2i_2 +i_3 \\  
 a_{i_4} & 2i_1 -2i_2 +2i_3- i_4 \\   
 \dots
 \end{array}
\]
If $a_{i_k}$ and $a_{i_j}$, with $k>j$ happen to be contributed to the same position, then $a_{i_k}$ appears to the left of $a_{i_j}$ (just as in $g$) in the corresponding section. More formally, 
\begin{align*}
 g &= \prod_{j=m}^1  a_{i_j}  = \prod_{j=m}^1  \mu_{i_j} \delta_{i_j}( a_{i_j}) = 
 \prod_{j=m}^1 \mu_{i_j} \cdot  \prod_{j=m}^1  \delta_{\mu_{i_1}\mu_{i_2} \dots \mu_{i_{j-1}} (i_j)}(a_{i_j}) \\
 &= \prod_{j=m}^1 \mu_{i_j} \cdot \prod_{j=m}^1  \delta_{2(i_1-i_2+i_3-i_4 +\dots+(-1)^j i_{j-1}) +(-1)^{j+1} i_j}(a_{i_j}), 
\end{align*}
where all the products should be written from left to right with decreasing indices. 
\end{lemma}

\begin{lemma}
The group $D$ is contracting, with nucleus $\{1,a_0,\dots,a_{d-1}\}$. Moreover,  for all $g \in D$ and $x \in X$, 
\[
 |g_x| \leq \frac{1}{2}(|g|+1),
\]
where the length function is with respect to the generating set $S=\{a_0,\dots,a_{d-1}\}$. In particular, for all $g \in G$ with $|g|\geq 2$, we have 
\[
 |g_x| < |g|. 
\]
\end{lemma}

%----------------------------------------
\subsection{The commutator subgroup $D'$}

For a word $W$ over $S=\{a_0,\dots,a_{d-1}\}$ and $j = 0,\dots,d-1$, let $\exp_j(W)$ be the number of occurrences of the letter $a_j$ in $W$. 

\begin{lemma}\label{lem:exp_def}
If the word $W$ represents the identity in $D$, then, for $j = 1,\dots,d-1$,   
\[
 \exp_j(W) \equiv_2 0. 
\]
\end{lemma}

\begin{proof}
By induction on $|W|$. The only word of length up to 1 that represents the identity is the empty word. This establishes the base of the induction. Let the claim be true for all words of length $\leq m$, for some $m \geq 1$, and let $W$ be a word of  length $m+1$ representing the identity in $D$. If two consecutive letters in $W$ are equal, we may remove them without affecting the parity of the exponents or the group element represented by $W$, obtain a shorter word, and apply the induction hypothesis. Otherwise, because of the contraction, the first level sections of $W$ can be represented by words $W_0,\dots,W_{d-1}$ of length $\leq m$. Since the sections of a trivial element are trivial, each of the word $W_0,\dots,W_{d-1}$ represents the identity. Therefore, by the induction hypothesis, for $j = 0,\dots,d-1$, 
\[
 \exp_j(W) = \sum_{i=0}^{d-1} \exp_j(W_i) \equiv_2 0. \qedhere
\]
\end{proof}

The last result shows that the exponents modulo 2 are well defined at the level of group elements by setting $\exp_j(g) = \exp_j(W)$, for $g \in G$ and $j=0,\dots,d-1$, where $W$ is any word over $S$ representing $g$. 

\begin{lemma}\label{lem:comm_is_ker_exp}
\begin{enumerate}
\item 
The map $\exp: G \to (\ZZ/2\ZZ)^d$ by $\exp(g) = (\exp_0(g),\dots,\exp_{d-1}(g))$ is a surjective homomorphism. 

\item The kernel of $\exp$ is the commutator of $D$. 

\item A word $W$ over $S$ represents an element in the commutator $D'$ if and only if $\exp_j(W) \equiv_2 0$, for $j =0,\dots,d-1$. 

\item $|D:D'| = 2^d$ and $D/D'$ is elementary abelian 2-group of rank $d$, generated by (the image) of $S$. 

\item $|D:\St_D(1)D'|=2$ and $\St_D(1)D'/D'$ is the kernel of the augmentation map
$\aug: D/D'\rightarrow \ZZ/2\ZZ$ defined by $a_0^{\alpha_0}\dots a_{d-1}^{\alpha_{d-1}} \mapsto \sum_{i=0}^{d-1}\alpha_i$.

\item $|\St_D(1) : \St_D(1)\cap D'|=|D'\St_D(1) : D'|=2^{d-1}$.
\end{enumerate}
\end{lemma}

\begin{proof}
Since the parity of $\exp_j$ does not depend on the representatives, we have $\exp_j(gh) \equiv_2 \exp_j(g)+\exp_j(h)$, which shows that $\exp$ is a homomorphism. Surjectivity is clear, since $\exp_j(a_j) = 1$. 

Let $K$ be the kernel of $\exp$. 
By definition of $\exp$, the element $g$ is in $K$ if and only if, for $j=0,\dots,d-1$, we have $\exp_j(g) \equiv_2 0$. 
In other words, an element $g$ represented by a word $W$ over $S$ is in $K$ if and only if every generator appears an even number of times in $W$. 

Since the image of $\exp$ is abelian, we have $D' \leq K$. On the other hand, since every generator of $D$ has order 2, every element of $D$ represented by a word in which every generator appears an even number of times is in the commutator $D'$. Thus, $K \leq D'$.  

Since $D'=K$ and $\exp$ is surjective, we have $D/D' = D/K = (\ZZ/2\ZZ)^d$. Note that, for $j=0,\dots,d-1$, the vector $\exp(a_j)$ is the $j$th standard basis vector of $(\ZZ/2\ZZ)^d$. 

%As $|D:\St_D(1)|=|D(d)|=2d$ and $d$ is odd, we have 
%$$|D:\St_D(1)D'|=\frac{2d}{|\St_D(1)D':\St_D(1)|}=\frac{|D:D'|}{|\St_D(1)D':D'|}=\frac{2^d}{|\St_D(1)D':D'|}=2.$$ 

As $D/\St_D(1)\cong D(d)$ and $d$ is odd, the abelianisation of $D(d)$ has order 2 and is isomorphic to $D/D'\St_D(1)$, where only the parity of the number of reflections matters.

The last item follows from the previous two.
\end{proof}

\begin{lemma}\label{lem:reg_branch_over_comm}
	The group $D$ branches over its commutator subgroup $D'$.
\end{lemma}
\begin{proof}

For $i \neq j$, the indices $i$, $2i-j$, and $3i-2j$ are distinct and, from Lemmas~\ref{lem:dihedral_relations} and \ref{lem:section_decompos_elts_hanoihedral} we obtain
\begin{align*}
	a_{2i-j}a_ia_ja_i  &= \delta_{3i-2j}(a_i)~\delta_{2i-j}(a_{2i-j}a_j)~ \delta_i(a_i) \\
	(a_{2i-j}a_ia_ja_i)^2 &= \delta_{2i-j}([a_{2i-j},a_j]). 
\end{align*} 
If $i = 0$, then
\begin{equation}\label{eq:subdirect_a0}
	(a_{2i-j}a_ia_ja_i)=\delta_{-2j}(a_{i})\delta_{-j}(a_{-j}a_j)\delta_0(a_0)=(a_0, *, *,\dots, *)
\end{equation}
where $*$ represents elements that are not important; while if $i\neq 0$, 
\begin{equation}\label{eq:subdirect_ai}
	(a_{2i-j}a_ia_ja_i)^{a_{i/2}} = \delta_0(a_i)\delta_{j-i}(a_{2i-j}a_j)\delta_{2(j-i)}(a_i)= (a_i,*,*,\dots,*). 
\end{equation}

If $2i-j = 0$, set $h = 1$, otherwise $h=a_{i-j/2}$. 
Then 
\begin{equation}\label{eq:branch_commutator}
	((a_{2i-j}a_ia_ja_i)^2)^h = ([a_{2i-j},a_j],1,1,\dots,1).
\end{equation}

By Lemma~\ref{lem:comm_is_ker_exp}, the elements $((a_{2i-j}a_ia_ja_i)^2)^h$ are all in $D'$, and letting $i, j$ run through $\{0, \dots, d-1\}$ in \eqref{eq:branch_commutator} yields that $D'$  contains the set $\{\delta_0([a_i,a_j]) : i\neq j\in \{0,\dots, d-1\} \}$.
Conjugating this set by appropriate products of elements \eqref{eq:subdirect_a0} and \eqref{eq:subdirect_ai}, we obtain that $D'\geq 0\ast D'$. 
Since $D$ is transitive on $X$, we can conjugate by elements that permute $X$ to get $D'\geq X\ast D'$.

%	For this, note that $$[a_j,a_{j+i}]=a_ja_{j+i}a_ja_{j+i}=\delta_j(a_ja_{j+2i})\delta_{j+i}(a_{j+i})\delta_{j-i}(a_{j+i})$$
%for $i, j\in\{1,\dots,d-1\}$,
%so $$[a_j,a_{j+i}]^2=(a_ja_{j+i}a_ja_{j+i})^2=\delta_{j}([a_j,a_{j+2i}]).$$ 
%Since $d$ is odd, division by 2 is possible (equivalently, $2i$ cycles through all of $\{1,\dots, d-1\}$) so that, choosing $i$ appropriately in the above expression ($i=(k-j)/2$), we obtain
%$$[a_{j},a_{(j+k)/2}]^2=\delta_j([a_j,a_k])\in D'$$ for all $j,k\in\{0,\dots,d-1\}, j\neq 0, k\neq j$. 
%
%Since $\mu_{(j+l)/2}$ swaps $j$ and $l$, conjugating $[a_{j},a_{(j+k)/2}]$ by $a_{(j+l)/2}$ yields
%$$a_{(j+l)/2}[a_{j},a_{(j+k)/2}]^2a_{(j+l)/2}=\delta_l([a_j,a_k])\in D'$$ for all $j,k,l\in\{0,\dots,d-1\}, j\neq 0, k\neq j$. 
%
%We have thus obtained that $\rst_{D'}(x)$ contains all possible commutators of generators of $D$, for all $x\in X$.
%
%Now, suppose that $\delta_l(c)\in D'$ for some $c\in D'$ and we wish to obtain $\delta_l(a_mca_m)\in D'$ for some $m\in\{0,\dots,d-1\}$. 
%Knowing that $\mu_{(l+m)/2}$ swaps $l$ and $m$ we get $$a_{(l+m)/2}a_ma_{(l+m)/2}=\mu_l\delta_{(3l-i)/2}(a_{(l+m)/2})\delta_l(a_m)\delta_{(l+m)/2}(a_{(l+m)/2}).$$ 
%Conjugating $\delta_l(c)$ by the above we get 
%\begin{align*}
%&a_{(l+m)/2}a_ma_{(l+m)/2}\delta_l(c)a_{(l+m)/2}a_ma_{(l+m)/2}\\
%&=\delta_{(l+m)/2}(a_{(l+m)/2})\delta_l(a_m)\delta_{(3l-m)/2}(a_{(l+m)/2})\delta_{l}(c)\delta_{(3l-m)/2}(a_{(l+m)/2})\delta_{l}(a_m)\delta_{(l+m)/2}(a_{(l+m)/2})\\
%&=\delta_l(a_mca_m)
%\end{align*}
%for $l,m\in\{0,\dots,d-1\}$. 
\end{proof}

%----------------------------------------
\subsection{The first level stabiliser $\St_D(1)$}

While it is possible to analyse the first level stabiliser directly, we find it suitable to use graph homology. 

For a simple, connected, undirected graph $\Gamma=(V,E)$, the edge space is the vector space of dimension $|E|$ (the coordinates are indexed by $E$) over the field with 2 elements. Fix a vertex $v_0$ in $V$. Each walk $w$ starting at $v_0$ in $\Gamma$ is represented in the edge space by the vector $(\#_e(w))_{e \in E}$, where $\#_e(w)$ is the number of times the edge $e$ is used, in either direction, in the walk $w$ (note that only the parity plays a role). The cycle space of $\Gamma$, based at $v_0$, is the subspace of the edge space spanned by the representations of the closed walks in $\Gamma$, starting end ending at $v_0$ (the cycle space does not depend on $v_0$, but we prefer to have a base point anyway). 

Let $T$ be a spanning tree of $\Gamma$ and $E'$ be the set of edges not on the tree $T$. The number of edges in $E'$ is $|E|-|V|+1$. For each edge $e \in E'$ pick, arbitrarily, one endpoint to be the origin $o(e)$ and the other to be the terminus $t(e)$, and  let $p_e$ be unique path in $T$ from $v_0$ to $o(e)$, followed by the edge $e$ to $t(e)$, followed by the unique path in $T$ from $t(e)$ to $v_o$. The cycle space has dimension $|E|-|V|+1$ and the representatives of $p_e$, for $e \in E'$, in the edge space form a basis of the cycle space. 

The cycle space can also be described through vertex conditions as follows. A vector $(x_e)_{e \in E}$ is in the cycle space if and only if, for every vertex $v$, we have $\sum_{e\sim v} x_e = 0$, where the summation is taken over all edges incident with $v$ (all edges with either $o(e)=v$ or $t(e)=v$). 
Each of these $|V|$ vertex conditions is implied by the other vertex conditions, and any $|V|-1$ of them are linearly independent.
 Note that the vertex conditions simply state that $(x_e)_{e \in E}$ represents a closed walk, based at $v_0$, such that each edge $e$ in $E$ is used, up to parity, $x_e$ times, if and only if, for every vertex $v$, the number of distinct edges incident with $v$ that are used odd number of times is even (so that the number of entrances and exits from $v$ can be matched). 

In our situation, the graph $\Gamma$ is the left Schreier graph of the action of $D$ on its quotient $D(d)$,  with respect to the generating set $S=\{a_0,\dots,a_{d-1}\}$ of $D$.
The graph $\Gamma$ is isomorphic to the complete bipartite graph $K_{d,d}$. 
The $d$ rotation vertices $\rho_j$, $j=0,\dots,d-1$ are connected by a total of $d^2$ edges to the $d$ mirror symmetry vertices $\mu_j$, $j=0,\dots,d-1$.
For every edge $e$, we declare the rotation vertex incident to $e$ to be the origin, and the mirror symmetry vertex incident to $e$ to be the terminus. 

For a fixed rotation vertex $\rho_j$, there are $d$ edges, labelled by the generators in $S$, using $\rho^j$ as origin and connecting $\rho^j$ to the appropriate mirror symmetry. Since $\mu_i \rho^j = \mu_{i-j/2}$ the edge labelled by $a_i$ with origin $\rho^j$ looks like 
\begin{equation}\label{e:edge}
\xymatrix{
 &  \ar@{}[l]|{\rho^j} \bullet \ar@{-}[r]^{a_i} & \bullet \ar@{}[r]|{\mu_{i-j/2}} &
}
\end{equation}
Similarly, for a fixed mirror symmetry vertex $\mu_j$, there are $d$ edges, labelled by the generators in $S$, using $\mu_j$ as terminus and connecting $\mu_j$ to the appropriate rotation. Since $\mu_i\mu_j = \rho^{2(i-j)}$ the edge labelled by $a_i$ with terminus $\mu_j$ looks like 
\[
\xymatrix{
 &  \ar@{}[l]|{\rho^{2(i-j)}} \bullet \ar@{-}[r]^{a_i} & \bullet \ar@{}[r]|{\mu_{j}} &
}
\]

We interpret $(X\ast D)/(X\ast D') = X\ast (D/D') \cong (\ZZ/2\ZZ)^{d^2}$ as the edge space of $\Gamma$ as follows. 
We take the trivial rotation $1=\rho^0$ to be the base of $\Gamma$.
 A word $W=a_{i_m} \dots a_{i_1}$ over $S$ is identified with the corresponding walk in $\Gamma$ starting at $1$ (the first step is along the edge labelled by $a_{i_1}$, the next along the edge labelled by $a_{i_2}$, and so on). When we pass along the edge $e$ in~\eqref{e:edge} in the direction from the origin $\rho^j$, the letter $a_i$ is contributed to position $\rho^{-j}(i) = i-j$ in the decomposition of $g=a_{i_m} \dots a_{i_1}$. When we pass along the same edge in the direction from the terminus, the letter $a_i$ is contributed to position $\mu_{i-j/2}^{-1}(i)= \mu_{i-j/2}(i) = i-j$. Thus, $a_i$ is contributed to the same position, $i-j$, regardless of the direction in which we pass the edge $e$. 
 This means that we can identify $(X\ast D)/(X\ast D') \cong (\ZZ/2\ZZ)^{d^2}$ with the edge space of $\Gamma$ by observing that, for the edge $e$ in~\eqref{e:edge}, a word $W=a_{i_m} \dots a_{i_1}$ and $g \in D$ represented by $W$, we have $\#_e(W) = \exp_i(g_{i-j})$.
  In other words, the parity of the letter $a_i$ in the $(i-j)$th component of the decomposition of $g$, taken as an element of $(X\ast D)/(X\ast D')\cong(\ZZ/2\ZZ)^{d^2}$, keeps track of the parity of $\#_e$. 

The first level stabiliser $\St_D(1)$ consist of all group elements represented by words which are closed walks based at $1$ in $\Gamma$. Thus, the image of $\St_D(1)$ in the edge space $(X\ast D)/(X\ast D')$ is precisely the cycle space of $\Gamma$. 

\begin{proposition}\label{prop:D1}
\begin{enumerate}
\item An element $a_{i_m}a_{i_{m-1}} \dots a_1$ belongs to the first level stabiliser $\St_D(1)$ if and only if $m$ is even and 
\[
 i_1+i_3+ \dots \equiv_d i_2+i_4+ \dots~.
\]

\item\label{i:D1_gens} The first level stabiliser $\St_D(1)$ is generated by the elements 
\[
 a_j a_{j+i} a_i a_0 = \delta_{j}(a_j)~\delta_{j-i}(a_{j+i})~\delta_{-i}(a_i)~\delta_{0}(a_0), 
\]
for $i,j =1,\dots,d-1$. 

\item\label{i:D1_equations}The elements of the first level stabiliser can be described as follows.
 An element $(g_0,\dots,g_{d-1})$ of $X\ast D$ is an element of $\St_D(1)$ if and only if, for $j = 0, \dots,d-1$, 
\[
 \sum_{i=0}^{d-1} \exp_{j+i}(g_i) \equiv_2 \sum_{i=0}^{d-1} \exp_{j-i}(g_i) \equiv_2 0.
\]
Each of these $2d$ parity conditions is implied by the other $2d-1$. 

\item\label{i:D1_index} $|X\ast D:\St_D(1)| = 2^{2d-1}$ and $|\St_D(1):X\ast D'| = 2^{(d-1)^2}$. Both $(X\ast D)/\St_D(1)$ and $\St_D(1)/(X\ast D')$ are elementary abelian 2-groups, of ranks $2d-1$ and $(d-1)^2$, respectively. 
\end{enumerate}
\end{proposition}

\begin{proof}
(i) The root permutation of $a_{i_m} \dots a_{i_1}$ is $\mu_{i_m} \dots \mu_{i_1}$ and $a_{i_m} \dots a_{i_1}$ represents an element in $\St_D(1)$ if and only if $\mu_{i_m} \dots \mu_{i_1}=1$, which happens precisely when $m$ is even (so that we have a rotation) and $i_m - i_{m-1} + \dots + i_2 - i_1 \equiv_d 0$ (so that the rotation is trivial). 

(ii) Let $T$ be the spanning tree of $\Gamma$ consisting of all edges incident with the trivial rotation $1$ together with all edges incident with $\mu_0$. For $i,j = 1\dots,d-1$, the cycles 
\[ 
\xymatrix{
 1  \ar@{-}[r]^{a_0} & 
 \mu_0 \ar@{-}[r]^{a_i} & 
 \rho^{2i} \ar@{-}[r]^{a_{i+j}} & 
 \mu_j \ar@{-}[r]^{a_{j}}&
 1 
}
\]
form a basis of the cycle space (the third edge is the one not in $T$ and it connects the non-trivial rotation vertex $\rho^{2i}$ to the mirror symmetry vertex $\mu_j$, not equal to $\mu_0$). Thus, the elements $a_ja_{i+j}a_ia_0$ generate $\St_D(1)$.

(iii) In our situation, we have $2d$ vertex conditions, one for every vertex in the Cayley graph of $D(d)$. 

Fix a vertex $\rho^j$ representing a rotation. Its incident edges are $a_{j+i}$, for $i = 0,\dots,d-1$. The edge $a_{j+i}$ is contributed to position $\rho^{-j}(j+i) = i$. The vertex condition at $\rho^j$ then reads
\[
 \sum_{i=0}^{d-1} \exp_{j+i}(g_i) \equiv_2 0. 
\]
Fix a vertex $\mu_{j/2}$ representing a mirror symmetry. Its incident edges are $a_{j-i}$, for $i = 0,\dots,d-1$. The edge $a_{j-i}$ is contributed to position $\mu_{j/2}^{-1}(j-i) = \mu_{j/2}(j-i)=i$. The vertex condition at $\mu_{j/2}$ then reads
\[
 \sum_{i=0}^{d-1} \exp_{j-i}(g_i) \equiv_2 0. 
\]

(iv) The dimension of the cycle space $\St_D(1)/(X\ast D')$ is $(d-1)^2$ since we have $|E|-|V|+1 = d^2-2d+1=(d-1)^2$. The dimension of the complementary space $(X\ast D)/\St_D(1)$ is $|V|-1 = 2d-1$. 
\end{proof}

%--------------------------------------------------
%\subsection{The rigid level stabilisers of $D$.}

\begin{proposition}\label{prop:rist_D}
	The rigid level stabiliser $R_n=\rst_D(n)$ of level $n$ in $D$ is $X^n\ast D'$.
	For $n \geq 1$, we have 
	\[
	|D:R_n| = 2^{(d-2)d^n+2} \cdot d^{\frac{d^{n}-1}{d-1}}. 
	\]
\end{proposition}
\begin{proof}
	
	We have already seen in Lemma \ref{lem:reg_branch_over_comm} that $x\ast D'\leq D'$ and therefore $x\ast D'\leq \rst_D(x)$ for every $x\in X$. 
	This implies that $X\ast D'\leq R_1$. 
	Inductively, $X^n\ast D'\leq X^{n-1}\ast D'\leq D$ for every $n\geq 1$ and so $X^n\ast D'\leq R_n$ for every $n\geq 1$.

	To show the opposite containment, consider an element $g=\delta_x(g_x)\in \rst_D(x)$ for some $x\in X$. 
	Then, since  $g\in \St_D(1)$, it must satisfy the equations in Proposition~\ref{prop:D1} \ref{i:D1_equations}, which reduce to $\exp_j(g_x)\equiv_2 0$ for all $j\in\{0,\dots,d-1\}$. 
	By Lemma \ref{lem:comm_is_ker_exp}, $g_x\in D'$, so $\rst_{D'}(x)\leq \rst_D(x)\leq x\ast D'$ for every $x\in X$ and $X\ast D'= R_1=\rst_{D'}(1)$.
	
	From this we obtain that $X^2\ast D'= X\ast \rst_{D'}(1)=R_2$ and inductively that $X^n\ast D'=R_n$ for all $n\geq 1$. 
	
	By the above,  Proposition \ref{prop:D1}\ref{i:D1_index}, and since $D/\St_D(1)\cong D(d)$, we have 
	$$|D:R_1|=|D:\St_D(1)||\St_D(1):(X\ast D')|=2d\cdot 2^{(d-1)^2}. $$
	This in turn implies that 
	$$|D':R_1|=|D':(X\ast D')|=\frac{|D:R_1|}{|D:D'|}=d2^{(d-1)^2+1-d}=d\cdot2^{(d-1)(d-2)}.$$
	Since $|D:R_1|=|D:D'||D':R_1|$ and $R_i=X\ast R_{i-1}$, so $R_i=X^i\ast D'$ for all $i$,  we have 
	$$|D:R_n|=|D:R_1|\times \prod_{i=1}^{n-1}|R_i:R_{i+1}|=%|D:D'|\times|D':R_1|\times \prod_{i=1}^{n-1}|R_i:R_{i+1}|=
	|D:D'|\times\prod_{i=0}^{n-1}|D':X\ast D'|^{d^i}=2^{s}\cdot d^t$$
	where  $t=\sum_{i=0}^{n-1}d^i=(d^n-1)/(d-1)$ and  $s=d+(d-1)(d-2)\sum_{i=0}^{n-1}d^i =d^n(d-2)+2$.
\end{proof}

%-----------------------------------------------

\subsection{The rigid kernel of $D$}

We first prove that the hypotheses of Theorem \ref{thm:rigid_kernel_surjective_system_direct_prod} hold for $D$. 
According to Proposition~\ref{prop:rist_D}, the maximal branching subgroup of $D$ is $D'$ and we know from \ref{lem:comm_is_ker_exp} that  $D/D'$ is an elementary abelian 2-group of rank $d$. 
This means that $D$ satisfies the hypotheses of the second item of the theorem. 
It only remains to see that $\St_D(2)D'\geq \St_D(1)$ and find $\St_D(2)/\Triv_D(2)$.

\begin{lemma}\label{lem:D2modT2}
	The following items hold:
	\begin{enumerate}
		\item  $(\St_D(2)(X\ast D'))/(X\ast D') \cong \St_D(2)/\Triv_D(2)$.
		\item $(\St_D(2)(X\ast D'))/(X\ast D')$ is an elementary abelian 2-group of rank $(d-2)(d-1)$ with a basis given by the images modulo $(X\ast D')$ of 
		\begin{equation}\label{eq:gens_D2_mod_XD'}
			a_ja_{j+i}a_ia_0\cdot a_{\frac{j}{2}}a_0a_{\frac{-j}{2}}a_0\cdot a_{\frac{j-i}{2}}a_0a_{\frac{i-j}{2}}a_0\cdot a_{\frac{-i}{2}}a_0a_{\frac{i}{2}}a_0, \quad j,i\in\{1,\dots,d-1\}, i\neq -j. 
		\end{equation}
		\item The patterns of the corresponding generators of $\St_D(2)/\Triv_D(2)$ are:
		\[%
		\xymatrix@!C=4em{
			&& a_0a_ia_{j+i}a_{\frac{j}{2}}a_{\frac{-j}{2}} a_{\frac{j-i}{2}}a_{\frac{i-j}{2}} a_{\frac{-i}{2}}a_{\frac{i}{2}} 
			\ar@{->}|-{j}[lld] \ar@{->}|-{j-i}[ld] \ar@{->}|-{-i}[d] \ar@{->}|-{j/2}[rd] \ar@{->}|-{(j-i)/2}[rrd] \ar@{->}|-{-i/2}[rrrd] &&&& \\
			a_{\frac{j}{2}}a_ja_{\frac{j}{2}}a_{0}&  a_{\frac{j+i}{2}}a_{j+i}a_{\frac{j+i}{2}}a_0 & a_{\frac{i}{2}}a_{i}a_{\frac{i}{2}}a_0 & a_{\frac{j-i}{2}}a_0a_{\frac{i-j}{2}}a_0 & a_{\frac{j}{2}}a_0a_{\frac{-j}{2}}a_0 & a_{\frac{i}{2}}a_0a_{\frac{-i}{2}}a_0 %\\
			% & & &  & & 
		}
		\] %
		
		\item 	$\St_D(2)D'=\St_D(1)D'$ and $\{\St_D(n)/\Triv_G(n), r_{m,n}\}$ is a surjective inverse system. 
	\end{enumerate}
\end{lemma}
\begin{proof}
	
	\begin{enumerate}[leftmargin=\parindent]
		\item Since $\Triv_D(2)=X\ast\Triv_D(1)=X\ast(\St_D(1)\cap D')$ and $\St_D(2)=\St_D(1)\cap (X\ast \St_D(1))$,  we have 
		$$\frac{\St_D(2)}{\Triv_D(2)}=\frac{\St_D(1)\cap X\ast\St_D(1)}{X\ast \St_D(1)\cap X\ast D'}\cong \frac{\St_D(1)\cap (X\ast \St_D(1)D')}{X\ast D'} =\frac{\St_D(2)(X\ast D')}{X\ast D'},$$ 

		\item From the above, we must work out  $\frac{\St_D(1)\cap (X\ast \St_D(1)D')}{(X\ast D')}$.
		Recall from Proposition \ref{prop:D1} \ref{i:D1_gens} that $\St_D(1)/(X\ast D')$ is generated by $(d-1)^2$ elements: $$a_ja_{j+i}a_ia_0=\delta_j(a_j)\delta_{j-i}(a_{j+i})\delta_{-i}(a_i)\delta_0(a_0)\quad \text{ for } i,j\in\{1,\dots,d-1\}.$$
		%
		%As seen above, $\St_D(2)(X\ast D')=\St_D(1)\cap (X\ast \St_D(1)D')$.
		Recall from Lemma~\ref{lem:comm_is_ker_exp} that $\St_D(1)D'/D'$ is the kernel of the augmentation map
		$\aug: D/D'\rightarrow \ZZ/2\ZZ$ defined by $a_0^{\alpha_0}\dots a_{d-1}^{\alpha_{d-1}} \mapsto \sum_{i=0}^{d-1}\alpha_i.$
		Thus, $\frac{\St_D(1)\cap (X\ast \St_D(1)D')}{(X\ast D')}$ is the kernel of the map 
		$$\Aug:\St_D(1)/(X\ast D')\rightarrow (\ZZ/2\ZZ)^d, \quad (x_0, \dots, x_{d-1})\mapsto (\aug(x_0),\dots,\aug(x_{d-1})).$$ 
		
		%The size of this quotient is 
		%$$\frac{|D:\Triv_D(2)|}{|D:\St_D(2)|}=\frac{2^{d^2-2d+2}d^{d+1}}{2^dd^{d+1}}=2^{d^2-3d+2}.$$
		%So $\frac{\St_D(2)}{\Triv_D(2)}$ is an elementary abelian 2-group of rank $d^2-3d+2$. 

		%
		
		The parity conditions given in Proposition~\ref{prop:D1}\ref{i:D1_equations} when added all together imply that $(x_0, \dots, x_{d-1})\in \St_D(1)/(X\ast D')$ satisfies the equation $\sum_{i=0}^{d-1}\aug(x_i)\equiv_2 0$ and therefore the image of $\Aug$ has dimension at most $d-1$ in $\ZZ/2\ZZ^d$. 
%		
%		Examining the $2d$ parity conditions given in Proposition \ref{prop:D1}\ref{i:D1_equations}, we see that in the image of $\Aug$, they reduce to the single condition $\sum_{i=0}^{d-1}\aug(x_i)\equiv_2 0$.
		Now note that, as $d$ is odd, $2j$ ranges through all of $\{1,\dots,d-1\}$ with $j\in\{1,\dots,d-1\}$. 
		So the $d-1$ images 
		$$\Aug(a_ja_0a_{-j}a_0)= \Aug(\delta_j(a_ja_{-j})\delta_{2j}(a_0)\delta_0(a_0))=\delta_{2j}(1)\delta_0(1)$$
		for $j\in\{1,\dots,d-1\}$ are linearly independent and so they form a basis for the image of $\Aug$. 
		Consequently, the kernel of $\Aug$ is an elementary abelian 2-group of rank $(d-1)^2-(d-1)=(d-2)(d-1)$.
		
		Expressing the images under $\Aug$ of the generators $\{a_ja_{j+i}a_ia_0: i,j=1,\dots,d-1\}$ of $\St_D(1)$ in terms of the above basis elements yields 
		$\Aug(a_ja_{j+i}a_ia_0)=\Aug(a_{-i}a_{-i-j}a_{-j}a_0)$  and so the $(d-1)(d-2)$ relations
		$$\Aug(a_ja_{j+i}a_ia_0)=\Aug(a_{j/2}a_0a_{-j/2}a_0)+\Aug(a_{(j-i)/2}a_{0}a_{(i-j)/2}a_0)+\Aug(a_{-i/2}a_0a_{i/2}a_0),$$
		for $i,j\in\{1,\dots,d-1\}, i\neq -j$ form a basis for the kernel of $\Aug$.
		This basis consists of the images, modulo $X\ast D'$,  of 
		$$a_0a_ia_{j+i}a_j\cdot a_{j/2}a_0a_{-j/2}a_0\cdot a_{-i/2}a_0a_{i/2}a_0\cdot a_{(j-i)/2}a_0a_{(i-j)/2}a_0, \quad j,i=1,\dots,d-1, i\neq -j.$$

		\item Using that $a_{j}a_{j+i}a_ia_0=\delta_j(a_j)\delta_{j-i}(a_{j+i})\delta_{-i}(a_i)\delta_0(a_0)$, it is straightforward to see that an element $g$ of the form shown in \eqref{eq:gens_D2_mod_XD'} has the following non-trivial sections: $g_j=a_ja_0$, $g_{-i}=a_ia_0$, $g_{j-i}=a_{j+i}a_0$ and $g_k=a_ka_{-k}$ for $k=j/2, -i/2, (j-i)/2$. 
		In general, these sections are not in $\St_D(1)$, as needed for $g$ to be in $\St_D(2)$; they are only in $\St_D(1)D'$. 
		However, since $X\ast D'\leq \Triv_D(1)$, we can find appropriate elements of $D$ to multiply $g$ and obtain an element of $\St_D(2)$ with the same pattern at the root.
		For example, we can use the following elements with appropriate values of  $k\in\{1,\dots, d-1\}$ 
		$$a_ka_0\cdot a_0[a_k, a_{k/2}]a_0=a_{k/2}a_ka_{k/2}a_0\in \St_D(1),  \qquad a_ka_{-k}[a_{-k},a_0]=a_ka_0a_{-k}a_0 \in \St_D(1)$$
		to obtain the patterns in the statement. 
		
		\item		Choosing $j=i$ in \eqref{eq:gens_D2_mod_XD'} and taking the image modulo $D'$, we obtain the generators $\{a_ja_0D': j=1,\dots d-1\}$ which are exactly the generators of $\St_D(1)D'/D'$. 
		Thus $\St_D(2)D'\geq \St_D(1)$ and Lemma \ref{lem:surjective_system_iff_G2K=G1K} yields that the inverse system $\{\St_D(n)/\Triv_D(n): n\geq 1\}$ is surjective.

	\end{enumerate}
\end{proof}

In particular, by the first part of Theorem \ref{thm:rigid_kernel_surjective_system_direct_prod}, $D$ has non-trivial rigid kernel. 
It only remains to find the structure of $\frac{\St_D(2)\cap D'}{\Triv_D(2)}$ to conclude, using the second item of the Theorem, that the rigid kernel of $D$ is an infinite Cartesian product of finite groups.

\begin{lemma}\label{lem:D2capD'modT2}
	The quotient $\frac{\St_D(2)\cap D'}{\Triv_D(2)}$ is an elementary abelian 2-group of rank $(d-1)(d-3)$. 
	Writing $[j,i]:=a_ja_{j+i}a_ia_0\in\St_D(1)$ for $j,i\in\{1,\dots, d-1\}$, the following patterns of size 2 generate $\frac{\St_D(2)\cap D'}{\Triv_D(2)}$:
	\begin{multline*}
		\delta_0([j/2,-j/2]\cdot[(j-i)/2, -(j-i)/2]\cdot [-i/2,i/2]\cdot[j,i] )%
		\cdot%
		\\%
		\delta_j([j/2,j/2])%
		\delta_{j-i}([(j+i)/2,(j+i)/2])%
		\delta_{-i}([i/2,i/2])%		
		\delta_{(j-i)/2}([(j-i)/2,-(j-i)/2])%
		\cdot%
		\\%
		\delta_{\pm j/2}([\mp j/4, \mp j/4])%
		\delta_{\pm i/2}([\pm i/2,\pm i/2])%
		\delta_{\pm (j+i)/2}([(j+i)/4,(j+i)/4])%
		\cdot%
		\\
		\delta_{\pm (j-i)/4}([(j-i)/4,-(j-i)/4])%
		\delta_{\pm (j+i)/4}([(j+i)/4, -(j+i)/4])
	\end{multline*}
	for $j,i\in\{1,\dots,d-1\}, i\neq \pm j$.
\end{lemma}

\begin{proof}
	As in the proof of Lemma \ref{lem:D2modT2}, we will work out the quotient $\frac{\St_D(2)(X\ast D')\cap D'}{X\ast D'}$, as it is isomorphic to $\frac{\St_D(2)\cap D'}{\Triv_D(2)}$.
%First note that $(\St_D(2)\cap D')/\Triv_D(2)=(\St_D(2)\cap D')(\St_D(2)\cap (X\ast D'))\cong (\St_D(2)(X\ast D')\cap D')/(X\ast D')$, so we will work out the latter quotient.
As $\St_D(2)(X\ast D')\cap D'=\St_D(2)(X\ast D')\cap \St_D(1)\cap D'$, we consider $\Aug$ from Lemma \ref{lem:D2modT2} and another map whose kernel is $\St_D(1)\cap D'/(X\ast D')$. 

Consider $g=(g_0, \dots, g_{d-1})\in \St_D(1)$.
By Lemma \ref{lem:comm_is_ker_exp}, $g\in D'$ if and only if each generator $a_j$ of $D$ appears an even number of times in $g$. 
Now, by the definition of the $a_j$, any appearance of $a_j$ in $g$ gives exactly one appearance of $a_j$ in some $g_i$.
So $g\in D'$ if and only if, for every $j\in X$, the number of appearances of $a_j$ across all $g_i$ is even ($\forall j\in X: \sum_{i=0}^{d-1}\exp_j(g_i)\equiv_2 0$, where $\exp_i$ is as in Lemma \ref{lem:exp_def}). 
In other words, $\St_D(1)\cap D'$ is the kernel of the map 
$$\mathrm{Exp}:\St_D(1)\cap (X\ast D) \rightarrow (\ZZ/2\ZZ)^d, \, (g_0, \dots, g_{d-1})\mapsto (\sum_{i=0}^{d-1}\exp_0(g_i), \dots, \sum_{i=0}^{d-1}\exp_{d-1}(g_i)).$$ 
By Lemma \ref{lem:D2modT2},  $\St_D(2)(X\ast D')/(X\ast D')$ is the kernel of the map $\Aug$, so 
$$\frac{(\St_D(2)(X\ast D')\cap \St_D(1)\cap D')}{(X\ast D')}$$
 is the kernel of the map 
%
%$\Aug\times \mathrm{Exp}: \St_D(1)/(X\ast D') \rightarrow (\ZZ/2\ZZ)^{2d}$ defined by
\begin{align*}
	\Aug\times \mathrm{Exp}: \St_D(1)/(X\ast D') &\rightarrow (\ZZ/2\ZZ)^{2d}\\
\mathbf{g}=(g_0,\dots, g_{d-1})&\mapsto (\Aug(\mathbf{g}), \mathrm{Exp}((\mathbf{g}))\\
&=((\aug(g_0),\dots,\aug(g_{d-1})),\, (\sum_{i=0}^{d-1}\exp_0(g_i), \dots, \sum_{i=0}^{d-1}\exp_{d-1}(g_i)) ).
\end{align*}
%where $\exp_i$ is as in Lemma \ref{lem:exp_def}. 

The $2d$ parity conditions given in Proposition \ref{prop:D1}\ref{i:D1_equations} imply that the sum of entries in each of the images of $\Aug$ and $\mathrm{Exp}$ must be 0 modulo 2. 
Thus the image of $\Aug\times \mathrm{Exp}$ has dimension at most $2d-2$. 

To improve readability, we write $[j,i]:=a_ja_{j+i}a_ia_0$, $\delta_n$ for the $n$th canonical basis vector in $(\ZZ/2\ZZ)^d$ and ${-}$ for the zero vector in $(\ZZ/2\ZZ)^d$.
Notice that the elements   
\begin{equation}\label{eq:basis_augexp}
\left\{
\begin{aligned}
 (\Aug\times\mathrm{Exp})([j,j])&=(\delta_j\delta_{-j},\; \delta_0\delta_{2j}), \quad j\in\{1,\dots, d-1\},\\
( \Aug\times\mathrm{Exp})([j,-j])&= (\delta_0\delta_{2j},\: \delta_{j}\delta_{-j}), \quad j\in\{1,\dots, d-1\}
\end{aligned}
\right.\end{equation}
are linearly independent, as $d$ is odd.
%
%Suppose that $\sum_{i=1}^{d-1}\alpha_i(\delta_i+\delta_{-i}, \delta_0+\delta_{2i})+\sum_{i=1}^{d-1}\beta_i(\delta_0+\delta_{2i}, \delta_i+\delta_{-i})=({-},{-})$.
%Comparing the $\delta_0$ coefficient in the first and second coordinates, respectively, we obtain that $\sum_{i=1}^{d-1}\alpha_i=\sum_{i=1}^{d-1}\beta_i=0$
%Looking at the $\delta_i$ coefficient in the 1st coordinate, $0=\alpha_i+\alpha_{-i}+\beta_{i/2}$ while the $\delta_{-i}$ coefficient in the 1st coordinate gives $0=\alpha_{-i}+\alpha_i+\beta_{-i/2}$; so $\beta_{i/2}+\beta_{-i/2}=0$ for every $i=1, \dots, d-1$. 
%As $d$ is odd, division by 2 is a permutation of $\{1, \dots, d-1\}$, so we have that $\beta_i+\beta_{-i}=0$ for every $i=1, \dots, d-1$. 
%By the same argument using the second coordinate, $\alpha_i+\alpha_{-i}=0$ for all $i=1, \dots, d-1$.
%Examining the $\delta_{2i}$ coefficient in the 1st coordinate: $0=\alpha_{2i}+\alpha_{-2i}+\beta_i$, so, by the previous lines, $0=\beta_i$ for all $i$. 
%The same argument in the second coordinate gives that $\alpha_i=0$ for all $i$. 
%
Again because $d$ is odd, these elements span the image of $\Aug\times \mathrm{Exp}$:
\begin{align*}
	(\Aug\times\mathrm{Exp})([2j,-2j][j,j][-j,-j])&=(\delta_0\delta_{4j}, {-})\\
	(\Aug\times\mathrm{Exp})([2j,2j][j,-j][-j,j])&=({-}, \delta_0\delta_{4j})
\end{align*}
for $j\in \{1,\dots, d-1\}$.
 So the image of $\Aug\times \mathrm{Exp}$ is indeed of rank $(2d-2)$ and therefore the kernel has rank $(d-1)^2-(2d-2)=(d-1)(d-3)$.
%
%
%
%To see that the elements of \eqref{eq:basis_augexp} indeed generate the image of $\Aug\times \mathrm{Exp}$, observe that 
% \begin{align*}
% (\Aug\times\mathrm{Exp})([2j,-2j][j,j][-j,-j])&=(\delta_0\delta_{4j}, {-})\\
% (\Aug\times\mathrm{Exp})([2j,2j][j,-j][-j,j])&=({-}, \delta_0\delta_{4j})
% \end{align*}

For $j,i\in\{1,\dots, d-1\}, i\neq \pm j$, we have the following $(d-1)(d-3)$ relations:
\begin{align*}
(\Aug\times\mathrm{Exp})([j,i]) &= (\delta_0\delta_j\delta_{j-i}\delta_{-i}, \; \delta_0\delta_j\delta_i\delta_{j+i})\\
&=(\delta_0\delta_{j}, {-})+(\delta_0\delta_{j-i}, {-})+ (\delta_0\delta_{-i}, {-})  + ({-}, \delta_0\delta_{j})+({-},\delta_0\delta_{i})+({-},\delta_0\delta_{j+i}).
\end{align*}
 which yield generators of $\frac{\St_D(2)(X\ast D')\cap D'}{X\ast D'}$, the images modulo $X\ast D'$ of 
\begin{multline*}
[j,i]\cdot [j/2,-j/2][j/4,j/4][-j/4,-j/4]  \cdot [i/2,-i/2][i/4,i/4][-i/4,-i/4]\cdot \\
 [(j-i)/2,(i-j)/2][(j-i)/4,(j-i)/4][(i-j)/4,(i-j)/4]\cdot\\
[j/2,j/2][-j/4,j/4][-j/4,j/4]  \cdot [i/2,i/2][-i/4,i/4][i/4,-i/4]\cdot\\ [(j+i)/2,(j+i)/2][(j+i)/4,(-j-i)/4][(-j-i)/4,(j+i)/4]
\end{multline*}
for $j, i\in\{1,\dots,d-1\}, i\neq \pm j$.
Using the same trick as in the proof of Lemma \ref{lem:D2modT2}, we can multiply each of the above elements by a suitably chosen element of $(X\ast D')$ to ensure that the patterns at each $x\in X$ are in $\St_D(1)$. 
This gives the patterns in the statement.
%
%\begin{multline*} \delta_0(a_0a_ja_ia_{j+i}a_{\pm j/2}a_{\pm i/2}a_{\pm (j-i)/2})\delta_j(a_0a_j)\delta_{-i}(a_0a_i)\delta_{j-i}(a_0a_{j+i})\\
%\delta_{\pm j/2}(a_0a_{\mp j/2})\delta_{\pm i/2}(a_0a_{\pm i/2})\delta_{\pm (j+i)/2}(a_0a_{(j+i)/2})\delta_{(j-i)/2}(a_{\pm(j-i)/2})\\
%\delta_{\pm (j-i)/4}(a_{\pm (j-i)/4})\delta_{\pm (j+i)/4}(a_{\pm (j+i)/4})
%\end{multline*}
\end{proof}

Notice that for $d=3$, the quotient $\frac{\St_D(2)\cap D'}{\Triv_D(2)}$ is trivial and we have only $(d-1)(d-2)=2$ generators of $\St_D(2)$ modulo $X\ast D'$, namely $a_1a_{2}a_1a_0\cdot a_{2}a_0a_{1}a_0\cdot a_{1}a_0a_{2}a_0 \equiv(a_0a_2,a_0a_2,a_0a_2)$ and  $a_2a_{1}a_2a_0\cdot a_{1}a_0a_{2}a_0\cdot a_{2}a_0a_{1}a_0\equiv (a_0a_1,a_0a_1,a_0a_1)$.
Therefore the rigid kernel is isomorphic to the Klein-4 group, generated by these elements. 
This  agrees  with the conclusions of \cite{bartholdi-s-z:congruence,skipper:constructive}.

\begin{corollary}
	The rigid kernel of $D$ is isomorphic to a Cartesian product of cyclic groups of order 2. If $d=3$, the rank is 2; if $d\geq 5$, the rank is infinite.
\end{corollary}

 %-----------------------------------------------
 \subsection{Congruence completion and Hausdorff dimension of $D$}

 The Hausdorff dimension of a closed subset of $A^{X^*}$ is a way of measuring its `relative size'. 
 If $A=F\leq \Sym(X)$,  $d=|X|$,  and $G\leq F^{X^*}=\varprojlim_n F\wr\overset{n}{\dots}\wr F $ is a closed subgroup, its Hausdorff dimension is
	 \begin{equation}\label{eq:hdim_def}
  \mathrm{hdim}(G)=\liminf_{n\to \infty}\frac{\log_{|F|}|G:\St_G(n)|}{\log_{|F|}{F\wr\overset{n}{\dots}\wr F}}=\liminf_{n\to\infty} \frac{(d-1)\log_{|F|}|G:\St_G(n)|}{d^n-1}.
	\end{equation}
 
% 
% If $G\leq F^{X^*}$ is a closed self-similar group, branching over its stabiliser subgroup $G_s$  then one can show inductively (see \cite{sunic:pibonacci}) that for all $m\geq 0$: 
% \begin{equation}\label{eq:index_formula_regbranch}
% 	[G:G_{s+m}]=\left(\frac{[G:G_1]}{[X\ast G:G_1]}\right)^{\frac{d^m-1}{d-1}}[G:G_s]^{d^m}.
% \end{equation} 
% Therefore $\log_{|F|}[G:G_{s+m}]=(\alpha -\beta)\frac{d^m-1}{d-1}+\gamma d^m$
% where $[G:G_1]=|F|^{\alpha}$, $[X\ast G:G_1]=|F|^{\beta}$ and $[G:G_s]=|F|^{\gamma}$. 
% 
% Inserting this into the definition of Hausdorff dimension we obtain
% \begin{equation}\label{eq:hdim_formula}
% 	\mathrm{hdim}(G)=\liminf_{n\to \infty}\frac{d-1}{d^{m+s}-1}\left(\frac{(\alpha-\beta)(d^m-1)}{d-1}+\gamma d^m\right)=\frac{\alpha-\beta+\gamma(d-1)}{d^s}.
% \end{equation}
% 
 
 In performing these calculations for $\cl{D}$, the appropriate choice of $F\leq \Sym(X)$ to take is $D(d)$, so $|F|=2d$ in the formulas above.

 \begin{proposition}\label{prop:size_DmodDn}
 	We have, for $n \geq 1$,  
 	\[
 	|D:\St_D(n)| = |\cl{D}:\cl{\St_D(n)}| = 2^{d^{n-1}} \cdot d^{\frac{d^{n}-1}{d-1}}.
 	\]
 	and the Hausdorff dimension of $\cl{D}$ is 
 	\[
 	\mathrm{hdim}(\cl{D})= 1- \frac{1}{d} \cdot \frac{\log 2}{\log 2d}. 
 	\]	
 \end{proposition}
% \begin{proof}
% 	
% 	We have already seen in \ref{} that $\cl{D}$ branches over $\cl{\St_D(1)}$, so $s=1$ in \eqref{eq:index_formula_regbranch}. 
% 	Also $[\cl{D}:\cl{\St_D(1)}]=[D:\St_D(1)]=2d$ and $[X\ast \cl{D}:\cl{\St_D(1)}]=2$, so we obtain 
% 	$$[D:\St_D(n)]=[\cl{D}:\cl{\St_D(n)}]=\left(\frac{2d}{2}\right)^{\frac{d^{n-1}-1}{d-1}}(2d)^{d^{n-1}}=2^{d^n-1}\cdot d^{\frac{d^{n-1}-1}{d-1}+d^{n-1}}=2^{d^{n-1}} \cdot d^{\frac{d^{n}-1}{d-1}}.$$
% 	
% 	Since $s=1$ and $[\cl{D}:\cl{\St_D(1)}]=2d$, we have that $\alpha=\gamma =1$ in \eqref{eq:hdim_formula}, while $[X\ast \cl{D}:\cl{\St_D(1)}]=2=(2d)^{\beta}$ means that $\beta=\log_{2d}(2)$. 
% 	Therefore
% 	$$\mathrm{hdim}(\cl{D})=\frac{1-\log_{2d}(2)+(d-1)}{d}=1-\frac{\log_{2d}(2)}{d}.$$
% \end{proof}

\begin{proof}
	The first equality is a general fact about profinite closures. 
	We only need to show   that the sizes of $D/\St_D(n)$ are as claimed. 
	For this, we use that $$|D : \St_D(n)|=\frac{|D : \rst_D(n)|}{|\St_D(n) : \Triv_D(n)|\cdot|\Triv_D(n) : \rst_D(n)|}.$$
	The numerator has already been determined in Proposition \ref{prop:rist_D}. 
	For the denominator, we have, firstly 
	$$|\Triv_D(n) : \rst_D(n)|=\left|X^{n-1}\ast(\St_D(1)\cap D') : X^n\ast D' \right|=\left|\St_D(1)\cap D' : X\ast D'\right|^{d^{n-1}}.$$
	Since $|\St_D(1): X\ast D'|=2^{(d-1)^2}$ by Proposition \ref{prop:D1}\ref{i:D1_index} and 
	$|\St_D(1) : \St_D(1)\cap D'|=2^{d-1}$ by Lemma~\ref{lem:comm_is_ker_exp},
	we have that 
	$|\St_D(1)\cap D' : X\ast D'| = 2^{(d-1)(d-2)} $.
	
	Now, from the proof of Theorem \ref{thm:rigid_kernel_surjective_system_direct_prod} we extract that 
	\begin{align*}
	|\St_D(n) : \Triv_D(n) | & = \left|X^{n-2}\ast\left(\frac{\St_D(2)\cap D'}{\Triv_D(2)}\right)\right| \cdot |\St_D(n-1) : \Triv_D(n-1) |\\
	& = |\St_D(2)\cap D' : \Triv_D(2)|^a \cdot |\St_D(1) : \Triv_D(1)|
	\end{align*}
	 where $a= \sum_{i=0}^{n-2}d^i=(d^{n-1}-1)/(d-1)$. 
	Using Lemma \ref{lem:D2capD'modT2} and that $|\St_D(1) : \Triv_D(1)|=2^{d-1}$, we conclude that 
	$$|\St_D(n) : \Triv_D(n) | = (2^{(d-1)(d-3)})^a \cdot 2^{d-1}=2^{(d-3)(d^{n-1}-1)+ d-1}.$$
	
	Putting all of the above together, we get that $|D:\St_D(n)|=(2^sd^t)/(2^u\cdot 2^v)$
	where $t=(d^{n}-1)/(d-1)$, $s=d^n(d-2)+2$, $u=(d-3)(d^{n-1}-1)+ d-1$, $v=(d-1)(d-2)d^{n-1}$, 
	so $s-u-v=d^{n-1}$ and we get the first part of the statement. 
	
	For the Hausdorff dimension, the result follows from \eqref{eq:hdim_def} and
	\begin{multline*}
		\log_{2d}\left(|\cl{D}:\St_{\cl{D}}(n)|\right)= 
		\log_{2d} \left( |D:\St_D(n)|\right) 
		=\log_{2d}\left(2^{d^{n-1}} \cdot d^{\frac{d^{n}-1}{d-1}}\right)\\
		=\frac{d^n-1}{d-1} -  \frac{d^{n-1}-1}{d-1}\log_{2d}(2).
	\end{multline*}
%	Inserting this into \eqref{eq:hdim_def} yields the result.	
\end{proof}
 
 \begin{corollary}
 		\begin{enumerate}
 		\item The closure $\cl{D}$ is a finitely constrained group defined by the  patterns of size $2$ that can be described as follows. A pattern of size 2 
 		
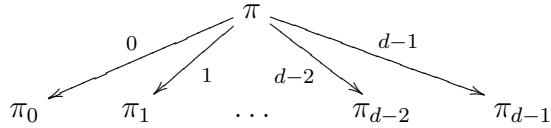
\[
 		\xymatrix{
 			&& \pi 
 			\ar@{->}_{0}[lld] \ar@{->}^{1}[ld] \ar@{->}_{d-2}[rd] \ar@{->}^{d-1}[rrd]&& \\
 			\pi_0 & \pi_1 & \dots & \pi_{d-2} & \pi_{d-1} 
 		}
 		\]
 		is an allowed pattern if and only if the permutation $\pi\pi_0\pi_1\dots \pi_{d-2}\pi_{d-1} \in D(d)$ is a rotation in $D(d)$ (the number of mirror symmetries among $\pi,\pi_0,\dots,\pi_{d-1}$ is even).

 		\item The closure $\cl{D}$ is a regular branch group branching over $\cl{\St_D(1)}=\St_{\cl{D}}(1)$, the stabiliser of level 1. 
 		Moreover, 
 		\[
 		[X\ast D :\St_D(1)] = 2^{2d-1} \qqand [X\ast\cl{D} :\cl{\St_D(1)}] =2. 
 		\]
 		so we recover, using Theorem \ref{t:criterion}, the known fact that $D$ has non-trivial rigid kernel. 
 	\end{enumerate}
 	
 \end{corollary}

\begin{proof}
	\begin{enumerate}
		\item First note that all elements of $D/\St_D(2)$ satisfy the constraints in the statement, because each generator $a_i$ of $D$ contributes a single reflection $\mu_i$ to the pattern at the root and the same reflection $\mu_i$ to the pattern at exactly one vertex of level 1.
		Since $D$ is self-similar, all patterns in $\cl{D}$ of size $n$ must also satisfy these constraints. 
		The fact that all allowed patterns in the statement of size $n$ arise in $\cl{D}$ follows from the fact that $|\cl{D} : \St_{\cl{D}}(n)|$ as calculated in Proposition \ref{prop:size_DmodDn} is exactly the number of allowed patterns. 
		
		\item The first part of the statement follows from Theorem \ref{thm:ps_fincons_regbranch} and the fact that $\cl{D}$ is finitely constrained, defined by patterns of size 2. 
		For the second part, the first equality is proved in Proposition \ref{prop:D1} while the second follows from the constraint of index 2 in the allowed patterns of size 2. 
	\end{enumerate}
\end{proof}

\bibliographystyle{alpha}
\bibliography{rigid-kernel}

\def\cprime{$'$}
\begin{thebibliography}{FAGUA17}

\bibitem[BSZ12]{bartholdi-s-z:congruence}
Laurent Bartholdi, Olivier Siegenthaler, and Pavel Zalesskii.
\newblock The congruence subgroup problem for branch groups.
\newblock {\em Israel J. Math.}, 187:419--450, 2012.

\bibitem[CRW17]{CRW-Part2}
Pierre-Emmanuel Caprace, Colin~D. Reid, and George~A. Willis.
\newblock Locally normal subgroups of totally disconnected groups. {P}art
  {I}{I}: {C}ompactly generated simple groups.
\newblock {\em Forum Math. Sigma}, 5:e12, 2017.

\bibitem[CSCF{\v{S}}13]{cech-c-f-s:cellautomata}
Tullio Ceccherini-Silberstein, Michel Coornaert, Francesca Fiorenzi, and Zoran
  {\v{S}}uni{\'c}.
\newblock Cellular automata between sofic tree shifts.
\newblock {\em Theoretical Computer Science}, 506:79 -- 101, 2013.

\bibitem[FAGUA17]{agu_ggs}
Gustavo~A. Fern{\'a}ndez-Alcober, Alejandra Garrido, and Jone Uria-Albizuri.
\newblock On the congruence subgroup property for {GGS}-groups.
\newblock {\em Proc. Am. Math. Soc.}, 145(8):3311--3322, 2017.

\bibitem[Gar16]{garridoCSP}
Alejandra Garrido.
\newblock On the congruence subgroup problem for branch groups.
\newblock {\em Israel J. Math.}, 216(1):1--13, 2016.

\bibitem[GN{\v S}06]{grigorchuk-n-s:oberwolfach2}
Rostislav Grigorchuk, Volodymyr Nekrashevych, and Zoran {\v S}uni\'c.
\newblock Hanoi towers group on 3 pegs and its pro-finite closure.
\newblock {\em Oberwolfach Reports}, 25:15--17, 2006.

\bibitem[Gri80]{grig_burnside}
Rostislav~I. Grigorchuk.
\newblock On {B}urnside's problem on periodic groups.
\newblock {\em Funktsional. Anal. i Prilozhen.}, 14(1):53--54, 1980.

\bibitem[Gri85]{grig_intermediate}
R.~I. Grigorchuk.
\newblock Degrees of growth of finitely generated groups, and the theory of
  invariant means.
\newblock {\em Mathematics of the USSR-Izvestiya}, 25(2):259, 1985.

\bibitem[Gri00]{GrigNewHorizons}
Rostislav~I. Grigorchuk.
\newblock Just infinite branch groups.
\newblock In {\em New horizons in pro-{$p$} groups}, volume 184 of {\em Progr.
  Math.}, pages 121--179. Birkh\"auser Boston, Boston, MA, 2000.

\bibitem[G{\v{S}}06]{grigorchuk-s:hanoi-cr}
Rostislav Grigorchuk and Zoran {\v{S}}uni{\'k}.
\newblock Asymptotic aspects of {S}chreier graphs and {H}anoi {T}owers groups.
\newblock {\em C. R. Math. Acad. Sci. Paris}, 342(8):545--550, 2006.

\bibitem[G{\v{S}}07]{grigorchuk-s:standrews}
Rostislav Grigorchuk and Zoran {\v{S}}uni{\'c}.
\newblock Self-similarity and branching in group theory.
\newblock In {\em Groups {S}t. {A}ndrews 2005. {V}ol. 1}, volume 339 of {\em
  London Math. Soc. Lecture Note Ser.}, pages 36--95. Cambridge Univ. Press,
  Cambridge, 2007.

\bibitem[Neu86]{neumann:pride}
Peter~M. Neumann.
\newblock Some questions of {E}djvet and {P}ride about infinite groups.
\newblock {\em Illinois J. Math.}, 30(2):301--316, 1986.

\bibitem[Per07]{pervova:completions}
Ekaterina Pervova.
\newblock Profinite completions of some groups acting on trees.
\newblock {\em J. Algebra}, 310(2):858--879, 2007.

\bibitem[P{\v{S}}19]{penland-sunic:kitchens}
Andrew Penland and Zoran {\v{S}}uni{\'c}.
\newblock A language hierarchy and kitchens-type theorem for self-similar
  groups.
\newblock {\em Journal of Algebra}, 537:173 -- 196, 2019.

\bibitem[Ski19]{skipper:constructive}
Rachel Skipper.
\newblock A constructive proof that the {H}anoi towers group has non-trivial
  rigid kernel.
\newblock {\em Topology Proc.}, 53:1--14, 2019.

\bibitem[Ski20]{skipper:congruence}
Rachel Skipper.
\newblock The congruence subgroup problem for a family of branch groups.
\newblock {\em Internat. J. Algebra Comput.}, 30(2):397--418, 2020.

\bibitem[{\v{S}}un07]{sunic:pibonacci}
Zoran {\v{S}}uni{\'c}.
\newblock Hausdorff dimension in a family of self-similar groups.
\newblock {\em Geom. Dedicata}, 124:213--236, 2007.

\bibitem[Wil71]{wilsonJIclassification}
J.~S. Wilson.
\newblock Groups with every proper quotient finite.
\newblock {\em Proc. Cambridge Philos. Soc.}, 69:373--391, 1971.

\bibitem[Wil00]{wilsonNewhorizons}
John~S. Wilson.
\newblock On just infinite abstract and profinite groups.
\newblock In Marcus Sautoy, Dan Segal, and Aner Shalev, editors, {\em New
  Horizons in pro-p Groups}, volume 184 of {\em Progress in Mathematics}, pages
  181--203. Birkhäuser Boston, 2000.

\bibitem[Wil04]{Wilson_nonuniform}
John~S. Wilson.
\newblock On exponential growth and uniformly exponential growth for groups.
\newblock {\em Invent. Math.}, 155(2):287--303, 2004.

\end{thebibliography}

\end{document}